\documentclass[11pt]{article}

\usepackage{microtype}
\usepackage{graphicx}
\usepackage{subfigure}
\usepackage{booktabs} 
\usepackage{natbib}
\usepackage{epsfig,epsf,fancybox}
\usepackage{amsmath}
\usepackage{mathrsfs}
\usepackage{amssymb}
\usepackage{color}
\usepackage{multirow}
\usepackage{paralist}
\usepackage{verbatim}
\usepackage{algorithm}
\usepackage{algorithmic}
\usepackage{galois}
\usepackage{boxedminipage}
\usepackage{accents}
\usepackage{stmaryrd}
\usepackage[bottom]{footmisc}

\usepackage{natbib}

\usepackage{url}
\usepackage[colorlinks,linkcolor=magenta,citecolor=blue, pagebackref=true,backref=true]{hyperref}
\renewcommand*{\backrefalt}[4]{%
    \ifcase #1 \footnotesize{(Not cited.)}%
    \or        \footnotesize{(Cited on page~#2.)}%
    \else      \footnotesize{(Cited on pages~#2.)}%
    \fi}

\newtheorem{theorem}{Theorem}[section]

\newtheorem{lemma}[theorem]{Lemma}
\newtheorem{proposition}[theorem]{Proposition}

\newtheorem{definition}{Definition}[section]

\newtheorem{remark}[theorem]{Remark}

\usepackage{hyperref}

\newcommand{\EE}{\mathbb{E}}

\newcommand{\Prob}{\textnormal{Prob}}

\newcommand{\x}{\mathbf x}

\newcommand{\sv}{\mathbf v}

\newcommand{\GCal}{\mathcal{G}}

\newcommand{\br}{\mathbb{R}}

\newcommand{\ba}{\begin{array}}
\newcommand{\ea}{\end{array}}

\newcommand{\FCal}{\mathcal{F}}

\newcommand{\XCal}{\mathcal{X}}

\newcommand{\NCal}{\mathcal{N}}

\begin{document}


\begin{center}

{\bf{\LARGE{Finite-Time Last-Iterate Convergence for Multi-Agent Learning in Games}}}

\vspace*{.2in}
{\large{
\begin{tabular}{cccc}
Tianyi Lin$^{\star, \ddagger}$ & Zhengyuan Zhou$^{\star, \square}$ & Panayotis Mertikopoulos$^\triangle$ & Michael I. Jordan$^{\diamond, \dagger}$ \\
\end{tabular}
}}

\vspace*{.2in}

\begin{tabular}{c}
Department of Electrical Engineering and Computer Sciences$^\diamond$ \\
Department of Industrial Engineering and Operations Research$^\ddagger$ \\
Department of Statistics$^\dagger$ \\ 
University of California, Berkeley \\
Stern School of Business, New York University and IBM Research$^\square$ \\
Univ. Grenoble Alpes, CNRS, Inria, LIG, 38000 Grenoble and Criteo AI lab$^\triangle$
\end{tabular}

\vspace*{.2in}

\today

\vspace*{.2in}

\begin{abstract} 
In this paper, we consider multi-agent learning via online gradient descent in a class of games called $\lambda$-cocoercive games, a fairly broad class of games that admits many Nash equilibria and that properly includes unconstrained strongly monotone games. We characterize the finite-time last-iterate convergence rate for joint OGD learning on $\lambda$-cocoercive games; further, building on this result, we develop a fully adaptive OGD learning algorithm that does not require any knowledge of problem parameter (e.g. cocoercive constant $\lambda$) and show, via a novel double-stopping time technique, that this adaptive algorithm achieves same finite-time last-iterate convergence rate as non-adaptive counterpart. Subsequently, we extend OGD learning to the noisy gradient feedback case and establish last-iterate convergence results--first qualitative almost sure convergence, then quantitative finite-time convergence rates-- all under non-decreasing step-sizes. To the best of our knowledge, we provide the first set of results that fill in several gaps of the existing multi-agent online learning literature, where three aspects -- finite-time convergence rates, non-decreasing step-sizes and fully adaptive algorithms have been unexplored before in the literature. 
\end{abstract}
\let\thefootnote\relax\footnotetext{$^\star$ Tianyi Lin and Zhengyuan Zhou contributed equally to this work.}
\end{center}

\vspace{-2em}\section{Introduction}
In its most basic incarnation, online learning~\citep{blum1998line, shalev2012online, hazan2016introduction} can be described as a feedback loop of the following form:
\begin{enumerate}
\item The agent interfaces with the environment by choosing an \emph{action} $a_t \in \mathcal{A} \subseteq \mathbf{R}^d$ (e.g., bidding in an auction, selecting a route in a traffic network).
\item The environment then yields a reward function $r_t(\cdot)$, and the agent obtains the reward $r_t(a_t)$ and receives some \emph{feedback} (e.g., reward function $r_t(\cdot)$, gradient $\nabla r_t(a_t)$, or reward $r_t(a_t)$), and the process repeats.
\end{enumerate}
As the reward functions $r_t(\cdot)$ are allowed to change from round to round, the standard metric that quantifies the performance of an online learning algorithm is that of regret~\citep{blum2007external}: at time $T$, the regret is the difference between $\max_{a\in A} \sum_{t=1}^T u_t(a)$, the total rewards achieved by the best fixed action in hindsight, and 
$\sum_{t=1}^T u_t(a_t)$, the total rewards achieved by the algorithm. 
In the rich online learning literature~\citep{Zin03, kalai2005efficient, SSS07, arora2012multiplicative, shalev2012online, hazan2016introduction}, perhaps the simplest algorithm that achieves the minimax-optimal regret guarantee is Zinkevich's online gradient descent (OGD), where the agent simply takes a gradient step (at current action) to form the next action, performing a projection if necessary. Due to its simplicity and strong performance, it is arguably one of the most widely-used algorithms in online learning theory and applications~\citep{Zin03, HAK07, quanrud2015online}.

At the same time, the most common instantiation of the above online learning model (where reward functions change arbitrarily over time) is multi-agent online learning: each agent is making online decisions in an environment that consists of other agents who are simultaneously making online decisions and whose actions impact the rewards of other agents; that is, each agent's reward is determined by an (unknown) game. Note that in multi-agent online learning, as other agents' actions change, each agent's reward function, when viewed solely as a function of its own action, also changes, despite the fact that the underlying game mechanism is fixed. Consequently, in this setting, the universality of the OGD regret bounds raises high expectations in terms of performance guarantees, leading to the following fundamental question in game-theoretical learning~\citep{cesa2006prediction, shoham2008multiagent, viossat2013no, bloembergen2015evolutionary, monnot2017limits}: \emph{Would OGD learning, and more broadly no-regret learning, lead to Nash equilibiria?}

As an example, if all users of a computer network individually follow some no-regret learning algorithm (e.g. OGD) to learn the best route for their traffic demands, would the system eventually converge to a stable traffic distribution, or would it devolve to perpetual congestion as users ping-pong between different routes (like commuters changing lanes in a traffic jam)? Note that whether the process converges at all pertains to the stability of the joint learning procedure, while whether it converges to Nash equilibria pertains to the rationality thereof: if the learning procedure converges to a non-Nash equilibrium, then each can do better by not following that learning procedure.

\paragraph{Related work.} Despite the seeming simplicity, the existing literature has only provided scarce and qualitative answers to this question. This is in part due to the strong convergence mode conveyed by the question: while a large literature exists on this topic, much of them focuses on time-average convergence (i.e. convergence of the time average of the joint action), rather than last-iterate convergence (i.e. convergence of the joint action). 
However, not only is last-iterate convergence theoretically stronger and more appealing, it is also the only type of convergence that actually describes the system's evolution. This was a well-known point that was only recently rigorously illustrated in~\citet{mertikopoulos2018cycles}, where it is shown that even though follow-the-regularized-leader (another no-regret learning algorithm) converges to a Nash equilibrium in linear zero-sum games in the sense of time-averages, actual joint action orbits Nash equilibria in perpetuity. Motivated by this consideration, a growing literature~\citep{krichene2015convergence, lam2016learning, zhou2017mirror, NIPS2017_7169, ZMBG+17-NIPS, MBNS17, zhou2017stochastic, zhou2020convergence, zhou2018learning, zhou2018multi, Mertikopoulos-2019-Learning} has devoted the efforts to obtaining last-iterate convergence results. However, due to the challenging nature of the problem, all of those lat-iterate convergence results are qualitative. In particular, except in strongly monotone games (\citet{zhou2020robust} very recently established a $O(1/T)$ last-iterate convergence rate for OGD learning with noisy feedback\footnote{The perfect gradient feedback case has a last-iterate convergence of $O(\rho^T)$, for some $0<\rho < 1$ when the game is further Lipchitz. This follows from a classical result in variational inequality~\citep{facchinei2007finite}} in strongly monotone games), there are no quantitative, finite-time last-iterate convergence rates available\footnote{Except in convex potential games. In that case, the problem of converging to Nash equilbiria reduces to a convex optimization problem, where standard techniques apply}.

Additionally, an important element in multi-agent online learning is that the horizon of play is typically unknown. As a result, no-regret learning algorithms need to be employed with a decreasing learning rate (e.g., because of a doubling trick or as a result of an explicit $O(1/t^{\alpha})$ step-size tuning). In particular, in order to achieve last-iterate convergence-to-Nash results, all of the above mentioned work rest crucially on using decreasing step-size (often converging to $0$ no slower than a particular rate) in their algorithm designs. This, however, leads to the following general tenet: \emph{New information is utilized with decreasing weights}

From a rationality point of view, this is not only counter-intuitive \textendash\ it flies at the face of established economic wisdom. Instead of \emph{discounting} past information, players end up indirectly \emph{reinforcing} it by assigning negligible weight to recent observations relative to those in the distant past. This negative recency bias is unjustifiable from economic micro-foundations and principles, and it cannot reasonably account for any plausible model of human/consumer behavior. The above naturally raises another important open question, one that, if answered, can bridge the gap between online learning and rationalizable economic micro-foundations: \emph{Is no-regret learning without discounting recent information compatible with Nash equilibria?}

\paragraph{Contributions.} Reflecting on those two gaps simultaneously, we are thus led to the following ambitious research question, one that aims to close two open questions at once: \emph{Can we obtain finite-time last-iterate convergence rate using only non-decreasing step-size?} Our goal here is to make initial but significant progress in answering this question; our contributions are threefold.

First, we introduce a class of games that we call \emph{cocoercive} and which contain all strongly monotone games as a special case. We show that if each player adopts OGD, then the joint action sequence converges in last-iterate to the set of Nash equilibria at a rate of $o(1/T)$. The convergence speed more specifically refers to how fast the gradient norm squared converges to $0$: note that in cocoercive games, gradient norm converges to $0$ if and only if the iterate converges to the set of Nash equilibria. To the best of our knowledge, this is the first rate that provides finite-time last-iterate convergence that moves beyond the strong monotonicity assumption.

Second, we study in depth the stochastic gradient feedback case, where each player adopts OGD in $\lambda$-cocoercive games, with gradient corrupted by a zero-mean, martingale-difference noise, whose variance is proportional to current gradient norm squared as assumed in the relative random noise model~\citep{Polyak-1987-Introduction}. In this more challenging setting, we first establish that the joint action sequence converges in last-iterate to Nash equilibria almost surely under a constant step-size. The previous best such qualitative convergence is due to~\citet{Mertikopoulos-2019-Learning}, which shows that such almost sure convergence is guaranteed in a variationally stable game. Despite the fact that variationally stable games contain cocoercive games as a subclass, our result is not covered by theirs because~\citet{Mertikopoulos-2019-Learning} assumes compact action set, where we consider unconstrained action set--a more challenging scenario since the action iterates can \textit{a priori} be unbounded. Our result is further unique in that constant step-size is sufficient to achieve last-iterate almost sure convergence, while~\citet{Mertikopoulos-2019-Learning} requires decreasing step-size (that is square-summable-but-not-summable). Note that the relative random noise model is necessary for obtaining such constant step-size result: in an absolute random noise model (where the noise's second moment is bounded by a constant), the gradient descent iterate forms an ergodic and irreducible Markov chains, which induces an invariant measure that is supported on the entire action set, thereby making it impossible to obtain any convergence-to-Nash result. We then proceed a step further and characterize finite-time convergence rate. We establish two rates here: first, the expected time-average convergence rate is $O(1/T)$; second the expected last-iterate convergence rate is  $O(a(T))$, where $a(T)$ depends on how fast the relative noise proportional constants decrease to $0$. As a simple example, if those constants decrease to $0$ at an $O(1/\sqrt{t})$ rate, then the last-iterate convergence rate is $O(1/\sqrt{T})$. For completeness (but due to space limitation), we also present in the appendix a parallel set of results--last-iterate almost sure convergence, time-average convergence rate and last-iterate convergence rate--for the absolute random noise model (under diminishing step-sizes of course).
 
Third, and even more surprisingly, we provide--to the best of our knowledge--the first adaptive gradient descent algorithm that has last-iterate convergence guarantees on games. In particular, the online gradient descent algorithms mentioned above--both in the deterministic and stochastic gradient case--requires the cocoercive constant $\lambda$ to be known beforehand. Thus, this calls for adaptive variants that do not require such knowledge. In the deterministic setting, we design an adaptive gradient descent algorithm that operates without needing to know $\lambda$ and adaptively chooses its step-size based on past gradients. We then show that the same $o(\frac{1}{T})$ last-iterate convergence rate can be achieved with non-decreasing step-size. Previously, the closest existing result is~\citet{Bach-2019-Universal}, which provided an adaptive algorithm on variational inequality with time-average convergence guarantees. However, providing adaptive algorithms for last-iterate convergence is much more challenging and \citet{Bach-2019-Universal} further requires the knowledge of the diameter of domain set (which they assume to be compact) in their adaptive algorithm, whereas we operate in unbounded domains. Our analysis relies on a novel double stopping time analysis, where the first stopping time characterizes the first time until gradient norm starts to monotonically decrease, and the second stopping time, after the first stopping has occured, characterizes the first time the underlying pseudo-contraction mapping starts to rapidly converge. We also provide the adaptive algorithm in the stochastic gradient feedback setting and establish the same finite-time last-iterate convergence guarantee. Note that our results only imply convergence in unconstrained strongly monotone games. Constrained coercive games is another interesting setting that would require a different set of techniques and further exploration.

\section{Problem Setup}
In this section, we present the definitions of a game with continuous action sets, which serves as a stage game and provides a reward function for each player in an online learning process. The key notion defined here is called $\lambda$-cococercivity, which is weaker than $\lambda$-strong monotonicity and covers a wider range of games.  

\subsection{Basic Definition and Notation}
Throughout this paper, we focus on games played by a finite set of \textit{players} $i \in \NCal = \{1, 2, \ldots, N\}$. During the learning process, each player selects an \textit{action} $\x_i$ from a convex subset $\XCal_i$ of a finite-dimensional vector space $\br^{n_i}$ and their reward is determined by the profile $\x = (x_1, x_2, \ldots, x_N)$ of all players' actions. Throughout the paper, $\|\cdot\|$ denotes the Euclidean norm (in the corresponding vector space): other norms can be easily accomodated in our framework of course (and different $\XCal_i$'s can in general have different norms), although we will not bother with all of this since we do not plan to play with (and benefit from) complicated geometries. 
\begin{definition}
A continuous game is a tuple $\GCal = (\NCal, \XCal = \prod_{i=1}^N \XCal_i, \{u_i\}_{i=1}^N)$, where $\NCal$ is the set of $N$ players $\{1, 2, \ldots, N\}$, $\XCal_i$ is a convex set of some finite-dimensional vector space $\br^{n_i}$ representing the action space of player $i$, and $u_i: \XCal \rightarrow \br$ is the $i$-th player's payoff function satisfying:
\begin{enumerate}
\item For each $i \in \NCal$, the function $u_i(\x)$ is continuous in $\x$.
\item For each $i \in \NCal$, the function $u_i$ is continuously differentiable in $\x_i$ and the partial gradient $\nabla_{\x_i} u_i(\x)$ is Lipschitz continuous in $\x$.
\end{enumerate}
\end{definition}

The notation $\x_{-i}$ denotes the joint action of all players but player $i$. Consequently, the joint action $\x$ will frequently be written as $(x_i; \x_{-i})$. Two important quantities are specified as follows:
\begin{definition}
$\sv(\x)$ is the profile of the players' individual payoff gradients, i.e., $\sv(\x) = (v_1(\x), \ldots, v_N(\x))$, where $v_i(\x) \triangleq \nabla_{x_i} u_i(\x)$. 
\end{definition}
We are looking at pure-Nash equilibria, because we are studying continuous games, where the action set is already a finite-dimensional vector space, rather than a finite set as in the simpler finite games in which each player’s mixed strategy is a vector of probabilities in the simplex. In our setting, each action already lives in a continuum and we follow the standard definition of a pure Nash equilibrium.
\begin{definition}
$\x^* \in \XCal$ is called a (pure-strategy) Nash equilibrium of a game $\GCal$ if for each player $i \in \NCal$, it holds true that $u_i(x_i^*, \x_{-i}^*) \geq u_i(x_i, \x_{-i}^*)$ for each $x_i \in \XCal_i$. 
\end{definition}

\begin{proposition}\label{prop:existence}
In a continuous game $\GCal$, if $\x^* \in \XCal$ is a Nash equilibrium, then $(\x - \x^*)^\top\sv(\x^*) \leq 0$ for all $\x \in \XCal$. The converse also holds true if the game is concave: for each $i \in \NCal$, the function $u_i(x_i; \x_{-i})$ is concave in $x_i$ for all $\x_{-i} \in \prod_{j \neq i} \XCal_j$. 
\end{proposition}

Proposition~\ref{prop:existence} is a classical result (see also~\cite{Mertikopoulos-2019-Learning} for a proof) and shows that the Nash equilibria of a concave game are precisely the solutions of the variational inequality $(\x - \x^*)^\top\sv(\x^*) \leq 0$ for all $\x \in \XCal$. 

\begin{definition}
	A continuous game $\GCal$ is:
	\begin{enumerate}
		\item \textbf{montone} if $(\x' - \x)^\top(\sv(\x') - \sv(\x)) \leq 0$ for all $\x, \x' \in \XCal$.
		\item \textbf{strictly monotone} if $(\x' - \x)^\top(\sv(\x') - \sv(\x)) \leq 0$ for all $\x, \x' \in \XCal$, with equality if and only if $\x = \x'$. 
	\end{enumerate}
\end{definition}

\begin{remark}
We briefly highlight a few well-known properties of monotone and strictly monotone games.
Monotone games are automatically concave games: a concave game (and hence a monotone game) is guaranteed to have a Nash equilibrium when all the action sets $\mathcal{X}_i$ are convex and compact. 
Otherwise, particularly in the unconstrained setting (each $\XCal_i = \br^{n_i}$), a Nash equilibrium may not exist. Many results on monotone games can be read off from the variational inequality literature~\cite{facchinei2007finite}; see all~\cite{Mertikopoulos-2019-Learning} for a detailed discussion.

In a strictly monotone game (first introduced in~\citet{Rosen-1965-Existence} and referred to as diagonal strict concave games there), at most one Nash equilibrium exists; hence when all action sets are convex and compact, a strictly monotone game admits a unique Nash equilibrium. Additionally, when $\sv = \nabla f$ for some (smooth) function $f$, $f$ is strictly concave. The notion \textit{strictly} refers to the \textit{only if} requirement in the condition. Many useful results regarding strictly monotone games (under convex and compact action sets) can be found in~\citet{Rosen-1965-Existence}.
\end{remark}

Proceeding a step further, we can define strongly monotone games:

\begin{definition}
A continuous game $\GCal$ is called $\lambda$-strongly monotone if the payoff strongly monotone condition holds: $(\x' - \x)^\top(\sv(\x') - \sv(\x)) \leq -\lambda\|\x' - \x\|^2$ for all $\x, \x' \in \XCal$. 
\end{definition}

Note that strongly monotone games are a subclass of strictly monotone games. One appealing feature of strongly monotone games is that the finite-time convergence rate can be derived in terms of $\|\x_t - \x^*\|$, where $\x^*$ is the unique Nash equilibrium~\citep{zhou2020robust} (under convex and compact action sets). On the other hand, $\x_t$ can possibly converge to a limit cycle or repeatedly hit the boundary in monotone games~\cite{Mertikopoulos-2018-Cycles, Daskalakis-2018-Training} despite that the time-average $(\sum_{j=1}^t \x_j)/t$ converges. More recently, \citet{Mertikopoulos-2019-Learning} analyzed online mirror descent (OMD) learning (which contains OGD as a special case) in variational stable games (under convex and compact action sets) and proved that last-iterate convergence to Nash equilibria holds almost surely in the presence of imperfect feedback (i.e. gradient corrupted by an unbiased noise). This result is surprising since the notion of variational stablitly is much weaker than strict monotonicity (hence qualitative convergence to Nash in strictly monotone games is guaranteed), showing that strong monotonicity is unnecessary for last-iterate convergence of OGD learning.

However, there are no last-iterate convergence rates available for strictly monotone games (unconstrained or constrained), and such a result does not seem possible because the strictness gap can be made arbitrarily small (and yielding arbitrarily slow rates). In fact, without the quadratic growth of strong monotonicity, it seems impossible to attain the rate for $\|\x_t - \x^*\|$ and, moreover, using the method with a constant step-size is completely out of reach of the techniques of~\citet{zhou2020robust}. Finally, we remark that fully adaptive and parameter-free learning methods are also missing from game-theoretic analyses to date.

\subsection{$\lambda$-Cococercive Games}

\begin{definition}
A continuous game $\GCal$ is called $\lambda$-cocoercive if the payoff cocoercive condition holds: $(\x' - \x)^\top(\sv(\x') - \sv(\x)) \leq -\lambda\|\sv(\x') - \sv(\x)\|^2$ for all $\x, \x' \in \XCal$. 
\end{definition}
In this paper, we focus on the unconstrained $\lambda$-cocoercive game in which $\XCal_i = \br^{n_i}$, while the existing literature invariably assumes compactness, which is the constrained setting. The analysis for these two settings should be considered as complementary (and results in~\citet{zhou2020robust} would not apply to unconstrained strongly monotone cases). We study the unconstrained setting (which is another indication that new techniques different from previous work are needed) because the adaptive algorithms and analyses are easier to present. Three important comments are in order.

First, a cocoercive game is a monotone game (as can be easily seen from the definitions), but cocoercive games neither contain nor belong to strictly monotone games. 
When a Nash equilibrium exists in a cocoercive game, it may not be unique; further, all Nash equilibria of
a cococercive game shares the same individual payoff gradient (either constrained or unconstrained): neither of these two properties hold in a strictly monotone game.
For a simple one-player example where the cost function $f(x) = x^2$ for $x<0$ and $0$ otherwise: this game is cocoercive but not strictly monotone. Moreover, $(\x' - \x)^\top(\sv(\x') - \sv(\x)) = 0$ only implies that $\sv(\x') = \sv(\x)$ and $\x' = \x$ does not necessarily hold true.

Second, the unconstrained cocoercive games may not always have a Nash equilibrium since we lifted the compactness assumption. Accordingly, all of our subsequent convergence results are stated for games that do have Nash equilibria: we did so because we want our results to apply to all cocoercive games that have Nash equilibria. That said, many additional sufficient conditions can be imposed on a cocoercive game to ensure the existence of a Nash equilrium. One such sufficient condition is the coercivity of the costs: the costs go to infinity as joint actions go to infinity (as already alluded to, we didn’t assume both cocoercivity and coercivity because that would eliminate other cocoercive games that admit Nash equilibria, which is more restrictive).

Thirdly, we remark that it is more difficult to analyze the convergence property of online algorithms in unconstrained setting, especially when the feedback information is noisy, since the iterates are not necessarily assumed to be bounded. Further, since in a cocoercive game, $\x^* \in \XCal^*$ is a Nash equilibrium if and only if $\sv(\x^*) = 0$, the natural candidate for measuring convergence (i.e. optimality gap)
is  $\epsilon(\x) = \|\sv(\x)\|^2$. 

\subsection{Learning via Online Gradient Descent}
We describe the \textit{online gradient descent} (OGD) algorithm in our game-theoretical setting. Intuitively, the main idea is: At each stage, every player $i \in \NCal$ gets an estimate $\hat{v}_i$ of the individual gradient of their payoff function at current action profile, possibly subject to noise and uncertainty. Subsequently, they choose an action $x_i$ for the next stage using the current action and feedback $\hat{v}_i$, and continue playing.

Formally, starting with some arbitrarily (and possibly uninformed) iterate $\x_0 \in \br^n$ at $t=0$, the scheme can be described via the recursion  
\begin{equation}\label{Update:OGD-learning}
x_{i, t+1} \ = \ x_{i,t} + \eta_{t+1} \hat{v}_{i,t+1},  
\end{equation}
where $t \geq 0$ denotes the stage of process, $\hat{v}_{i,t}$ is an estimate of the individual payoff gradient $v_i(\x_t)$ of player $i$ at stage $t$. The learning rate $\eta_t > 0$ is a nonincreasing sequence which can be of the form $c/t^p$ for some $p \in [0, 1]$. 

\paragraph{Feedback and uncertainty.} We assume that each player $i \in \NCal$ has access to a ``black box" feedback mechanism -- an \textit{oracle} -- which returns an estimate of their payoff gradients at their current action profile. This information can be imperfect for a multitude of reasons; see~\citet[Section~3.1]{Mertikopoulos-2019-Learning}. With all this in mind, we consider the following noisy feedback model: 
\begin{equation}\label{model:noisy}
\hat{v}_{i,t+1} \ = \ v_i(\x_t) + \xi_{i, t+1}, 
\end{equation}
where the noise process $\xi_t = (\xi_{i, t})_{i \in \NCal}$ is an $L^2$-bounded martingale difference adapted to the history $(\FCal_t)_{t \geq 1}$ of $\x_t$ (i.e., $\xi_t$ is $\FCal_t$-measurable but $\xi_{t+1}$ isn't). 

We focus on two types of random noise proposed by~\cite{Polyak-1987-Introduction}. The first type is called \textbf{relative random noise}: 
\begin{equation}\label{model:noisy-relative}
\EE[\xi_{t+1} \mid \FCal_t] = 0, \quad \EE[\|\xi_{t+1}\|^2 \mid \FCal_t] \leq \tau_t\|\sv(\x_t)\|^2. 
\end{equation}
and the second type is called \textbf{absolute random noise}: 
\begin{equation}\label{model:noisy-absolute}
\EE[\xi_{t+1} \mid \FCal_t] = 0, \quad \EE[\|\xi_{t+1}\|^2 \mid \FCal_t] \leq \sigma_t^2,  
\end{equation}
The above condition is mild (the i.i.d. condition is not imposed) and allows for a broad range of error processes. For the relative random noise, the variance decreases as it approaches a Nash equilibrium which admits better convergence rate of learning algorithms.

\section{Convergence under Perfect Feedback}\label{sec:perfect_relative}
In this section, we analyze the convergence property of OGD learning under perfect feedback. In particular, we show that the finite-time last-iterate convergence rate is $o(1/T)$ regardless of fully adaptive learning rates. To our knowledge, the proof techniques for analyzing adaptive OGD learning is new and can be of independent interests.  

\subsection{OGD Learning}
We first provide a lemma which shows that $\|\sv(\x_t)\|^2$ is nonnegative, nonincreasing and summable. 
\begin{lemma}\label{Lemma:OGD-nonadaptive}
Fix a $\lambda$-cocoercive game $\GCal$ with continuous action spaces $(\NCal, \XCal = \prod_{i=1}^N \br^{n_i}, \{u_i\}_{i=1}^N)$ with an nonempty set of Nash equilibrium $\XCal^*$. Under the condition that $\eta_t = \eta \in (0, \lambda]$, the OGD iterate $\x_t$ satisfies for all $t \geq 0$  that $\|\sv(\x_{t+1})\| \leq \|\sv(\x_t)\|$ and 
\begin{eqnarray*}
\|\x_t - \Pi_{\XCal^*}(\x_0)\| & \leq & \|\x_0 - \Pi_{\XCal^*}(\x_0)\|, \\
\sum_{t=0}^{+\infty} \|\sv(\x_t)\|^2 & \leq & \frac{\|\x_0 - \Pi_{\XCal^*}(\x_0)\|^2}{\eta\lambda}.
\end{eqnarray*}
\end{lemma}
\begin{remark}
For the OGD learning with perfect feedback and constant step-size, the update formula in Eq.~\eqref{Update:OGD-learning} implies that $\|\sv(\x_t)\|^2 = \|\x_t - \x_{t+1}\|^2/\eta^2$. This implies that $\|\x_t - \x_{t+1}\|^2$ also serves as the candidate for an optimality gap function. Such quantity is called the \textit{iterative gap} and frequently used to construct the stopping criteria in practice. 
\end{remark}
Now we are ready to present our main results on the last-iterate convergence rate of OGD learning. 
\begin{theorem}\label{Theorem:OGD-nonadaptive}
Fix a $\lambda$-cocoercive game $\GCal$ with continuous action spaces $(\NCal, \XCal = \prod_{i=1}^N \br^{n_i}, \{u_i\}_{i=1}^N)$ with an nonempty set of Nash equilibrium $\XCal^*$. Under the condition that $\eta_t = \eta \in (0, \lambda]$, the OGD iterate $\x_t$ satisfies that $\epsilon(\x_T) = o(1/T)$. 
\end{theorem}
\begin{proof}
Lemma~\ref{Lemma:OGD-nonadaptive} implies that $\{\|\sv(\x_t)\|^2\|\}_{t \geq 0}$ is nonnegative, nonincreasing and $\sum_{t=0}^{+\infty} \|\sv(\x_t)\|^2 < +\infty$. Therefore,
\begin{equation*}
T\|\sv(\x_{2T-1})\|^2 \leq \sum_{t=T}^{2T-1} \|\sv(\x_t)\|^2 \ \rightarrow \ 0 \ \textnormal{as} \  T \rightarrow +\infty. 
\end{equation*}
which implies that $\|\sv(\x_T)\|^2 = o(1/T)$. By the definition of $\epsilon(\x)$, we conclude the desired result. 
\end{proof}
\begin{algorithm}[!t]
\caption{Adaptive Online Gradient Descent}\label{algorithm:AdaOGD}
\begin{algorithmic}[1]
\STATE \textbf{Initialization:} $\x_0 \in \br^n$, $\eta_1 = 1/\beta_1$ for some $\beta_1 > 0$ and tuning parameter $r > 1$.   
\FOR{$t = 0, 1, 2, \ldots$}
\FOR{$i = 1, 2, \ldots, N$}
\STATE $x_i^{t+1} \leftarrow x_i + \eta_{t+1}\sv(\x_t)$.
\ENDFOR 
\IF{$\|\sv(\x_{t+1})\| > \|\sv(\x_t)\|$}
\STATE $\beta_{t+2} \leftarrow r\beta_{t+1}$. 
\ELSE 
\STATE $\beta_{t+2} \leftarrow \beta_{t+1}$.   
\ENDIF
\STATE $\eta_{t+2} \leftarrow 1/\sqrt{\beta_{t+2} + \sum_{j=0}^t \|\sv(\x_j)\|^2}$. 
\ENDFOR
\end{algorithmic}
\end{algorithm}
\subsection{Adaptive OGD Learning}
Our main results in this subsection is the last-iterate convergence rate of Algorithm~\ref{algorithm:AdaOGD}. Here the algorithm requires no prior knowledge of $\lambda$ but still achieves the rate of $o(1/T)$. To facilitate the readers, we summarize the results in the following theorem and provide the detailed proof. 
\begin{theorem}\label{Thm:OGD-adaptive}
Fix a $\lambda$-cocoercive game $\GCal$ with continuous action spaces $(\NCal, \XCal = \prod_{i=1}^N \br^{n_i}, \{u_i\}_{i=1}^N)$ with an nonempty set of Nash equilibrium $\XCal^*$. The adaptive OGD iterate $\x_t$ satisfies that $\epsilon(\x_T) = o(1/T)$. 
\end{theorem}
\begin{proof}
Since the step-size sequence $\{\eta_t\}_{t \geq 1}$ is nonincreasing, we define the first iconic time in our analysis as follows,  
\begin{equation*}
t^* = \max\left\{t \geq 0 \mid \eta_{t+1} > \lambda\right\}. 
\end{equation*}
In what follows, we prove the last-iterate convergence rate for two cases: $t^* = +\infty$ (\textbf{Case I}) and $t^* < +\infty$ (\textbf{Case II}). 

\paragraph{Case I.} First, we have $1/\lambda^2 - \beta_0 \geq 0$ since $\eta_0 > \lambda$. Note that $\beta_{t+2} \leftarrow r\beta_{t+1}$ with $r > 1$ is updated when $\|\sv(\x_{t+1})\| > \|\sv(\x_t)\|$, there exists $T_0 > 0$ such that $\|\sv(\x_{t+1})\| \leq \|\sv(\x_t)\|$ for all $t \geq T_0$. If not, then $\beta_t \rightarrow +\infty$ as $t \rightarrow +\infty$ and $\eta_t \rightarrow 0$. However, $t^* = +\infty$ implies that $\eta_{t+1} \geq \lambda$ for all $t \geq 1$. This leads to a contradiction. Furthermore, it holds true for all $t \geq 0$ that
\begin{equation*}
\sum_{j = 0}^t \|\sv(\x_j)\|^2 \ \leq \ \frac{1}{\lambda^2} - \beta_{t+2} \ \leq \ \frac{1}{\lambda^2} - \beta_0 \ < \ +\infty.  
\end{equation*}
By starting the sequence at a later index $T_0$, we have $\sum_{t \geq T_0} \|\sv(\x_t)\|^2 < +\infty$ and $\|\sv(\x_{t+1})\| \leq \|\sv(\x_t)\|$ for all $t \geq T_0$. Using the same argument as in Theorem~\ref{Theorem:OGD-nonadaptive}, the adaptive OGD iterate $\x_t$ satisfies that $\epsilon(\x_T) = o(1/T)$. 

\paragraph{Case II.} First, we claim that $\|\x_t - \Pi_{\XCal^*}(\x_{t^*})\| \leq D$ where $D = \max_{1 \leq t \leq t^*} \|\x_t - \Pi_{\XCal^*}(\x_{t^*})\|$. Indeed, it suffices to show that $\|\x_t - \Pi_{\XCal^*}(\x_{t^*})\| \leq \|\x_{t^*} - \Pi_{\XCal^*}(\x_{t^*})\|$ holds for $t > t^*$. By the definition of $t^*$, we have $\eta_{t+1} \leq \lambda$ for all $t > t^*$. The desired inequality follows from Lemma~\ref{Lemma:OGD-nonadaptive}. 

Using the update formula (cf. Eq.~\eqref{Update:OGD-learning}), we have
\begin{equation*}
(\x_{t+1} - \x^*)^\top\sv(\x_t) \ = \ \frac{1}{2\eta_{t+1}}\left(\|\x_t - \x_{t+1}\|^2 + \|\x^* - \x_{t+1}\|^2 - \|\x^* - \x_t\|^2\right).
\end{equation*}
Using the update formula of OGD learning, we have
\begin{equation}\label{inequality-OGD-adaptive-second}
\lambda\|\sv(\x_t)\|^2 \ \leq \ \frac{1}{\eta_{t+1}}\left(\|\x^* - \x_t\|^2 - \|\x^* - \x_{t+1}\|^2\right) + \left(\frac{1}{\lambda} - \frac{1}{\eta_{t+1}}\right)\|\x_t - \x_{t+1}\|^2. 
\end{equation}
Summing up Eq.~\eqref{inequality-OGD-adaptive-second} over $t = 0, 1, 2, \ldots, T$ yields that 
\begin{equation*}
\sum_{t=0}^T \lambda\|\sv(\x_t)\|^2 \ \leq \ \sum_{t=1}^T \|\x^* - \x_t\|^2\left(\frac{1}{\eta_{t+1}} - \frac{1}{\eta_t}\right) + \frac{\|\x^* - \x_0\|^2}{\eta_1} + \sum_{t=0}^T \left(\frac{1}{\lambda} - \frac{1}{\eta_{t+1}}\right)\|\x_t - \x_{t+1}\|^2. 
\end{equation*}
Since the step-size sequence $\{\eta_t\}_{t \geq 1}$ is nonincreasing, we have $1/\eta_{t+1} \geq 1/\eta_t$. Letting $\x^* = \Pi_{\XCal^*}(\x_{t^*})$, we notice that $\|\x^* - \x_t\| \leq D$ for all $0 \leq t \leq T$. Putting these pieces together yields that 
\begin{equation*}
\sum_{t=0}^T \lambda\|\sv(\x_t)\|^2 \ \leq \ \frac{D^2}{\eta_{T+1}} + \sum_{t=0}^T \left(\frac{1}{\lambda} - \frac{1}{\eta_{t+1}}\right)\|\x_t - \x_{t+1}\|^2. 
\end{equation*}
To proceed, we define the second iconic time as 
\begin{equation*}
t_1^* \ = \ \max\left\{t \geq 0 \mid \eta_{t+1} > \frac{\lambda}{2D^2+1}\right\} \ > \ t^*.  
\end{equation*}
Suppose that $t_1^* = +\infty$, it is straightforward to show that the adaptive OGD iterate $\x_t$ satisfies that $\epsilon(\x_T) = o(1/T)$ using the same argument in \textbf{Case I}. 

Next, we consider $t_1^* < +\infty$. Indeed, we recall that $\eta_{t+1} \leq \lambda$ for all $t > t_1^*$ which implies that $1/\lambda - 1/\eta_{t+1} \leq 0$. Since $\|\x_t - \x_{t+1}\|^2 = \eta_{t+1}^2\|\sv(\x_t)\|^2$ (cf. Eq.~\eqref{Update:OGD-learning}) and assume $T$ sufficiently large without loss of generality, we have
\begin{equation*}
\sum_{t=0}^T \lambda\|\sv(\x_t)\|^2 \ \leq \ \frac{D^2}{\eta_{T+1}} + \sum_{t=0}^{t_1^*} \frac{\eta_{t+1}^2\|\sv(\x_t)\|^2}{\lambda} \ = \ \text{I} + \text{II}. 
\end{equation*}
Before bounding term $\text{I}$ and $\text{II}$, we present two technical lemmas which is crucial to our subsequent analysis; see~\citep[Lemma~A.1 and~A.2]{Bach-2019-Universal} for the detailed proof. 
\begin{lemma}\label{Lemma:OGD-adaptive-first}
For a sequence of numbers $a_0, a_1, \ldots, a_n \in [0, a]$ and $b \geq 0$, the following inequality holds: 
\begin{equation*}
\sqrt{b + \sum_{i=0}^{n-1} a_i} - \sqrt{b} \ \leq \ \sum_{i=0}^n \frac{a_i}{\sqrt{b + \sum_{j=0}^{i-1} a_j}} \ \leq \ \frac{2a}{\sqrt{b}} + 3\sqrt{a} + 3\sqrt{b + \sum_{i=0}^{n-1} a_i}. 
\end{equation*}
\end{lemma}
\begin{lemma}\label{Lemma:OGD-adaptive-second}
For a sequence of numbers $a_0, a_1, \ldots, a_n \in [0, a]$ and $b \geq 0$, the following inequality holds: 
\begin{equation*}
\sum_{i=0}^n \frac{a_i}{b + \sum_{j=0}^{i-1} a_j} \leq \ 2 + \frac{4a}{b} + 2\log\left(1 + \sum_{i=0}^{n-1} \frac{a_i}{b}\right). 
\end{equation*}
\end{lemma}
\paragraph{Bounding term $\text{I}$:} By the definition of $t_1^*$ and Lemma~\ref{Lemma:OGD-nonadaptive}, we have $\beta_t = \beta_{t_1^*+1}$ for all $t > t_1^*$. Thus, we derive from the definition of $\eta_t$ that 
\begin{equation*}
\text{I} \leq D^2\sqrt{\beta_{T+1} + \sum_{j=0}^{T-1} \|\sv(\x_j)\|^2} \leq D^2\sqrt{\beta_{t_1^*+1} + \sum_{j=0}^{T-1} \|\sv(\x_j)\|^2} 
\end{equation*}
Since $\|\x_t - \Pi_{\XCal^*}(\x_{t^*})\| \leq D$ for all $t \geq 0$. Since the notion of $\lambda$-cocercivity implies the notion of $(1/\lambda)$-Lipschiz continuity, we have
\begin{equation}\label{inequality-OGD-adaptive-bound}
\|\sv(\x_t)\|^2 \ \leq \ \frac{\|\x_t - \Pi_{\XCal^*}(\x_{t^*})\|^2}{\lambda^2} \ \leq \ \frac{D^2}{\lambda^2}. 
\end{equation}
Using the first inequality in Lemma~\ref{Lemma:OGD-adaptive-first}, we have
\begin{eqnarray}\label{inequality-OGD-adaptive-third}
\text{I} & \leq & D^2\sqrt{\beta_{t_1^*+1}} + \sum_{t=0}^T \frac{D^2\|\sv(\x_t)\|^2}{\sqrt{\beta_{t_1^*+1} + \sum_{j=0}^{t-1} \|\sv(\x_j)\|^2}}  \\
& \leq & D^2\sqrt{\beta_{t_1^*+1}} + \sum_{t=0}^{t_1^*} \frac{D^2\|\sv(\x_t)\|^2}{\sqrt{\beta_{t_1^*+1} + \sum_{j=0}^{t-1} \|\sv(\x_j)\|^2}} + \sum_{t=t_1^*+1}^T D^2\eta_{t+1} \|\sv(\x_t)\|^2. \nonumber 
\end{eqnarray}
Since $\eta_{t+1} \leq \lambda/2D^2$ for all $t > t_1^*$, we have
\begin{equation}\label{inequality-OGD-adaptive-fourth}
\sum_{t=t_1^*+1}^T D^2\eta_{t+1} \|\sv(\x_t)\|^2 \ \leq \ \sum_{t=t_1^*+1}^T \frac{\lambda\|\sv(\x_t)\|^2}{2}.
\end{equation}
Using the second inequality in Lemma~\ref{Lemma:OGD-adaptive-first}, 
\begin{equation}\label{inequality-OGD-adaptive-fifth}
\sum_{t=0}^{t_1^*} \frac{D^2\|\sv(\x_t)\|^2}{\sqrt{\beta_{t_1^*+1} + \sum_{j=0}^{t-1} \|\sv(\x_j)\|^2}} \ \leq \ \frac{2D^2}{\lambda^2\sqrt{\beta_{t_1^*+1}}} + \frac{3D}{\lambda} + 3\sqrt{\beta_{t_1^*+1} + \sum_{j=0}^{t_1^*-1} \|\sv(\x_j)\|^2}.
\end{equation}
By the definition of $\eta_t$, we have
\begin{equation}\label{inequality-OGD-adaptive-sixth}
\sqrt{\beta_{t_1^*+1} + \sum_{j=0}^{t_1^*-1} \|\sv(\x_j)\|^2} \ = \ \frac{1}{\eta_{t_1^*+1}} \ < \ \frac{2D^2+1}{\lambda}.  
\end{equation}
Putting Eq.~\eqref{inequality-OGD-adaptive-third}-\eqref{inequality-OGD-adaptive-sixth} together yields that 
\begin{equation*}
\text{I} \ \leq \ D^2\sqrt{\beta_{t_1^*+1}} + \frac{2D^2}{\lambda^2\sqrt{\beta_{t_1^*+1}}} + \frac{3 + 3D + 6D^2}{\lambda} + \sum_{t=t_1^*+1}^T \frac{\lambda\|\sv(\x_t)\|^2}{2}. 
\end{equation*}
\textbf{Bounding term $\text{II}$:} By the definition of $\eta_t$ and noting that $\beta_t \geq \beta_1$ for all $t \geq 1$, we have
\begin{equation*}
\text{II} \ \leq \ \frac{1}{\lambda}\left(\sum_{t=0}^{t_1^*} \frac{\|\sv(\x_t)\|^2}{\beta_1 + \sum_{j=0}^{t-1} \|\sv(\x_j)\|^2}\right)
\end{equation*}
Recalling Eq.~\eqref{inequality-OGD-adaptive-bound}, we can apply Lemma~\ref{Lemma:OGD-adaptive-second} with Eq.~\eqref{inequality-OGD-adaptive-sixth} to obtain that 
\begin{equation*}
\text{II} \ \leq \ \frac{1}{\lambda}\left(2 + \frac{4D^2}{\lambda^2\beta_1} + 2\log\left(1 + \frac{1}{\beta_1}\sum_{j=0}^{t_1^*-1} \|\sv(\x_j)\|^2\right)\right) \ \leq \ \frac{1}{\lambda}\left(2 + \frac{4D^2}{\lambda^2\beta_1} + 2\log\left(1 + \frac{8D^4+2}{\lambda^2\beta_1}\right)\right). 
\end{equation*}
Therefore, we conclude that 
\begin{equation*}
\sum_{t=0}^T \frac{\lambda\|\sv(\x_t)\|^2}{2} \ \leq \ D^2\sqrt{\beta_{t_1^*+1}} + \frac{2D^2}{\lambda^2\sqrt{\beta_{t_1^*+1}}} + \frac{3+3D+6D^2}{\lambda} + \frac{1}{\lambda}\left(2 + \frac{4D^2}{\lambda^2\beta_1} + 2\log\left(1 + \frac{8D^4+2}{\lambda^2\beta_1}\right)\right). 
\end{equation*}
which implies that $\sum_{t=0}^T \|\sv(\x_t)\|^2$ is bounded by a constant for all $T \geq 0$. By starting the sequence at a later index $t_1^*$, we have $\|\sv(\x_{t+1})\| \leq \|\sv(\x_t)\|$ for all $t \geq t_1^*$ and $\sum_{t \geq t_1^*} \|\sv(\x_t)\|^2 < +\infty$. Using the same argument as in Theorem~\ref{Theorem:OGD-nonadaptive}, we conclude that the adaptive OGD iterate $\x_t$ satisfies that $\epsilon(\x_T) = o(1/T)$. 
\end{proof}
\begin{remark}
In our proof, $D > 0$ is the constant that depends on the set of Nash equilibrium, and in stating the bound this way, we followed the standard tradition in optimization where the bound (on either $f(\x_t) - f(\x^*)$ or $\|\nabla f(\x^t)\|^2$) would depend on $\|\x_0 - \x^*\|$ (a constant that cannot be avoided). Here, our bound similarly depends on this constant, except the game setting is more complicated (since there is no common objective) so this constant depends on the first few iterates as well (not just the initial iterate $\x^0$): more precisely, $D = \max_{1 \leq t \leq t^*} \|\x_t - \Pi_{\XCal^*}(\x_{t^*})\|$. 
\end{remark}

\section{Convergence under Imperfect Feedback with Relative Random Noise}\label{sec:imperfect_relative}
In this section, we analyze the convergence property of OGD learning under imperfect feedback with relative random noise~\eqref{model:noisy-relative}. In particular, we show that the almost sure last-iterate convergence is guaranteed and the finite-time average-iterate convergence rate is $O(1/T)$ when $0 < \tau_t \leq \tau < +\infty$. More importantly, we get the finite-time last-iterate convergence rate when $\tau_t$ satisfies certain summable condition~\eqref{condition:noisy-OGD-relative}. 

\subsection{Almost Sure Last-Iterate Convergence}
In this subsection, we establish the almost sure last-iterate convergence under imperfect feedback with relative random noise. The appealing feature here is that the convergence results provably hold with a constant step-size. The first and second lemmas provide two different key inequalities for $\x_t$ and $\EE[\epsilon(\x_t)]$ respectively. 
\begin{lemma}\label{Lemma:key-inequality-noisy-OGD}
Fix a $\lambda$-cocoercive game $\GCal$ with continuous action spaces $(\NCal, \XCal = \prod_{i=1}^N \br^{n_i}, \{u_i\}_{i=1}^N)$ with an nonempty set of Nash equilibrium $\XCal^*$. Under the noisy model~\eqref{model:noisy}, the noisy OGD iterate $\x_t$ satisfies for any Nash equilibrium $\x^* \in \XCal^*$ that
\begin{equation}\label{inequality-noisy-OGD-absolute-key}
\|\x_{t+1} - \x^*\|^2 \ \leq \ \|\x_t - \x^*\|^2 + 2\eta_{t+1}^2\|\xi_{t+1}\|^2 - (2\lambda\eta_{t+1} - 2\eta_{t+1}^2)\|\sv(\x_t)\|^2 + 2\eta_{t+1}(\x_t - \x^*)^\top\xi_{t+1}.
\end{equation}
\end{lemma}
\begin{proof}
Using the update formula of $x_{i, t+1}$ in Eq.~\eqref{Update:OGD-learning}, we have the following for any $x_i^* \in \XCal_i^*$: 
\begin{equation*}
\|x_{i, t+1} - x_i^*\|^2 \ = \ \|x_{i, t} + \eta_{t+1} \hat{v}_{i, t+1} - x_i^*\|^2. 
\end{equation*}
which implies that 
\begin{equation*}
\|x_{i, t+1} - x_i^*\|^2 \ = \ \|x_{i, t} - x_i^*\|^2 + \eta_{t+1}^2\|\hat{v}_{i, t+1}\|^2 + 2\eta_{t+1}(x_{i, t} - x_i^*)^\top\hat{v}_{i, t+1}.   
\end{equation*}
Summing up the above inequality over $i \in \NCal$ and rearranging yields that
\begin{equation}\label{inequality-noisy-OGD-absolute-first}
\|\x_{t+1} - \x^*\|^2 \ = \ \|\x_t - \x^*\|^2 + 2\eta_{t+1}(\x_t - \x^*)^\top\hat{v}_{t+1} + \eta_{t+1}^2\|\hat{v}_{t+1}\|^2. 
\end{equation}
Using Young's inequality, we have
\begin{equation*}
\|\x_{t+1} - \x^*\|^2 \ \leq \ \|\x_t - \x^*\|^2 + 2\eta_{t+1}^2\|\xi_{t+1}\|^2 + 2\eta_{t+1}^2\|\sv(\x_t)\|^2 + 2\eta_{t+1}(\x_t - \x^*)^\top(\sv(\x_t) + \xi_{t+1}). 
\end{equation*}
Since $\x^* \in \XCal^*$ and $\GCal$ is a $\lambda$-cocoercive game, we have $\sv(\x^*) = 0$ and 
\begin{equation*}
(\x_t - \x^*)^\top\sv(\x_t) \ = \ (\x_t - \x^*)^\top(\sv(\x_t) - \sv(\x^*)) \ \leq \ -\lambda\|\sv(\x_t) - \sv(\x^*)\|^2 \ = \ -\lambda\|\sv(\x_t)\|^2. 
\end{equation*}
Putting these pieces yields the desired inequality. 
\end{proof}
\begin{lemma}\label{Lemma:OGD-noisy-relative-iterative}
Fix a $\lambda$-cocoercive game $\GCal$ with continuous action spaces $(\NCal, \XCal = \prod_{i=1}^N \br^{n_i}, \{u_i\}_{i=1}^N)$ with an nonempty set of Nash equilibrium $\XCal^*$. Under the noisy model~\eqref{model:noisy} with relative random noise~\eqref{model:noisy-relative} and the step-size sequence $\eta_t \in (0, \lambda)$. The noisy OGD iterate $\x_t$ satisfies $\EE[\epsilon(\x_{t+1})] \leq \EE[\epsilon(\x_t)] + \tau_t\|\sv(\x_t)\|^2/\lambda\eta_{t+1}$. 
\end{lemma}
\begin{proof}
Using the same argument as in Lemma~\ref{Lemma:key-inequality-noisy-OGD}, we have 
\begin{equation*}
\EE[\epsilon(\x_{t+1}) \mid \FCal_t] - \epsilon(\x_t) \ \leq \ \frac{\EE[\|\xi_{t+1}\|^2 \mid \FCal_t]}{\lambda\eta_{t+1}} + \left(1 - \frac{\lambda}{\eta_{t+1}}\right)\EE\left[\|\sv(\x_{t+1}) - \sv(\x_t)\|^2 \mid \FCal_t\right]. 
\end{equation*}
Since the noisy model~\eqref{model:noisy} is with relative random noise~\eqref{model:noisy-absolute}, we have $\EE[\|\xi_{t+1}\|^2 \mid \FCal_t] \leq \tau_t\|\sv(\x_t)\|^2$. Also, $\eta_t \in (0, \lambda)$ for all $t \geq 1$. Therefore, we conclude that 
\begin{equation*}
\EE[\epsilon(\x_{t+1}) \mid \FCal_t] - \epsilon(\x_t) \ \leq \ \frac{\tau_t\|\sv(\x_t)\|^2}{\lambda\eta_{t+1}}.  
\end{equation*}
Taking an expectation of both sides yields the desired inequality.
\end{proof}
Now we are ready to characterize the almost sure last-iterate convergence. Note that the condition imposed on $\tau_t$ is minimal and $\eta_t = \eta \in [\underline{\eta}, \overline{\eta}]$ is allowed for all $t \geq 1$. 
\begin{theorem}\label{Thm:noisy-OGD-relative-last}
Fix a $\lambda$-cocoercive game $\GCal$ with continuous action spaces $(\NCal, \XCal = \prod_{i=1}^N \br^{n_i}, \{u_i\}_{i=1}^N)$ with an nonempty set of Nash equilibrium $\XCal^*$. Under the noisy model~\eqref{model:noisy} with relative random noise~\eqref{model:noisy-relative} satisfying $\tau_t \in (0, \tau]$ for some $\tau < +\infty$ and the step-size sequence satisfying $0 < \underline{\eta} \leq \eta_t \leq \overline{\eta} < \lambda/(1+\tau)$ for all $t \geq 1$. The noisy OGD iterate $\x_t$ converges to $\XCal^*$ almost surely. 
\end{theorem}
\begin{proof}
We obtain the following inenquality by taking the expectation of both sides of Eq.~\eqref{inequality-noisy-OGD-absolute-key} (cf. Lemma~\ref{Lemma:key-inequality-noisy-OGD}) conditioned on $\FCal_t$: 
\begin{equation*}
\EE[\|\x_{t+1} - \x^*\|^2 \mid \FCal_t] \leq \|\x_t - \x^*\|^2 - (2\lambda\eta_{t+1} - 2\eta_{t+1}^2)\|\sv(\x_t)\|^2 + 2\eta_{t+1}^2\EE[\|\xi_{t+1}\|^2 \mid \FCal_t] + 2\eta_{t+1}\EE[(\x_t - \x^*)^\top\xi_{t+1} \mid \FCal_t]. 
\end{equation*}
Since the noisy model~\eqref{model:noisy} is with relative random noise~\eqref{model:noisy-relative} satisfying $\tau_t \in (0, \tau)$ for some $\tau < +\infty$, we have $\EE[(\x_t - \x^*)^\top\xi_{t+1} \mid \FCal_t] = 0$ and $\EE[\|\xi_{t+1}\|^2 \mid \FCal_t] \leq \tau\|\sv(\x_t)\|^2$. Therefore, we have
\begin{equation}\label{inequality-noisy-OGD-relative-first}
\EE[\|\x_{t+1} - \x^*\|^2 \mid \FCal_t] \ \leq \ \|\x_t - \x^*\|^2 - 2(\lambda - \overline{\eta} - \tau\overline{\eta})\eta_{t+1}\|\sv(\x_t)\|^2.
\end{equation}
Since $\eta_t > 0$ and $\overline{\eta} < \lambda/(1+\tau)$, we let $M_t = \|\x_t - \x^*\|^2$ and obtain that $M_t$ is an nonnegative supermartingale. Then Doob's martingale convergence theorem shows that $M_n$ converges to an nonnegative and integrable random variable almost surely. Let $M_\infty = \lim_{t \rightarrow +\infty} M_t$, it suffices to show that $M_\infty = 0$ almost surely now. We assume to contrary that, there exists $m > 0$ such that $M_\infty > m$ with positive probability. Then $M_t > m/2$ for sufficiently large $t$ with positive probability. Formally, there exists $\delta > 0$ such that
\begin{equation*}
\Prob(M_t > m/2 \ \textnormal{for sufficiently large } t) \ \geq \ \delta.  
\end{equation*} 
By the definition of $M_t$ and recalling that $\x^* \in \XCal^*$ can be any Nash equilibrium, we let $U$ be a $(m/2)$-neighborhood of $\XCal^*$ and obtain that $\x_t \notin U$ for sufficiently large $t$ with positive probability. Since $\|\sv(\x)\| = 0$ if and only if $\x \in \XCal^*$, there exists $c > 0$ such that $\|\sv(\x)\| \geq c$ for sufficiently large $t$ with positive probability. Therefore, we conclude that $\EE[\|\sv(\x_t)\|^2] \nrightarrow 0$ as $t \rightarrow +\infty$. 

On the other hand, by taking the expectation of Eq.~\eqref{inequality-noisy-OGD-relative-first} and using the condition $\eta_t \geq \underline{\eta} > 0$ for all $t \geq 1$, we have
\begin{equation*}
\EE[\|\sv(\x_t)\|^2] \ \leq \ \frac{\EE[\|\x_t - \x^*\|^2] - \EE[\|\x_{t+1} - \x^*\|^2]}{2(\lambda - \overline{\eta} - \tau\overline{\eta})\underline{\eta}}. 
\end{equation*}
This implies that $\sum_{t=0}^\infty \EE[\|\sv(\x_t)\|^2] < +\infty$ and hence $\EE[\|\sv(\x_t)\|^2] \rightarrow 0$ as $t \rightarrow +\infty$ which contradicts the previous argument. This completes the proof. 
\end{proof}

\subsection{Finite-Time Convergence Rate: Time-Average and Last-Iterate}
In this subsection, we focus on deriving two types of rates: the time-average and last-iterate convergence rates, as formalized by the following theorems.
\begin{theorem}\label{Thm:noisy-OGD-relative-average-rate}
Fix a $\lambda$-cocoercive game $\GCal$ with continuous action spaces $(\NCal, \XCal = \prod_{i=1}^N \br^{n_i}, \{u_i\}_{i=1}^N)$ with an nonempty set of Nash equilibrium $\XCal^*$. Under the noisy model~\eqref{model:noisy} with relative random noise~\eqref{model:noisy-relative} satisfying $\tau_t \in (0, \tau]$ for some $\tau < +\infty$ and a step-size sequence satisfying $0 < \underline{\eta} \leq \eta_t \leq \overline{\eta} < \lambda/(1+\tau)$ for all $t \geq 1$, the noisy iterate $\x_t$ satisfies $\frac{1}{T+1}(\EE[\sum_{t=0}^T \epsilon(\x_t)]) = O(1/T)$. 
\end{theorem}
\begin{proof}
Using the same argument as in Theorem~\ref{Thm:noisy-OGD-relative-last}, we obtain that
\begin{equation}\label{inequality-noisy-OGD-relative-first}
\EE[\|\x_{t+1} - \x^*\|^2 \mid \FCal_t] \ \leq \ \|\x_t - \x^*\|^2 - 2(\lambda - \overline{\eta} - \tau\overline{\eta})\eta_{t+1}\|\sv(\x_t)\|^2.
\end{equation}
Taking an expectation of both sides of Eq.~\eqref{inequality-noisy-OGD-relative-first} and rearranging yields that 
\begin{equation}\label{inequality-noisy-OGD-relative-second}
\EE[\epsilon(\x_t)] \ \leq \ \frac{1}{2(\lambda - \overline{\eta} - \tau\overline{\eta})\eta_{t+1}}\left(\EE[\|\x_t - \x^*\|^2] - \EE[\|\x_{t+1} - \x^*\|^2]\right). 
\end{equation}
Summing up the above inequality over $t=0, 1, \ldots, T$ yields that 
\begin{equation*}
\EE\left[\sum_{t=0}^T\epsilon(\x_t)\right] \ \leq \ \sum_{t=1}^T \left(\frac{1}{\eta_{t+1}} - \frac{1}{\eta_t}\right)\frac{\EE[\|\x_t - \x^*\|^2]}{2(\lambda - \overline{\eta} - \tau\overline{\eta})} + \frac{\|\x_0 - \x^*\|^2}{2(\lambda - \overline{\eta} - \tau\overline{\eta})\eta_1}. 
\end{equation*}
On the other hand, we have $\EE[\|\x_{t+1} - \x^*\|^2] \ \leq \ \EE[\|\x_t - \x^*\|^2]$. This implies that $\EE[\|\x_t - \x^*\|^2] \leq \|\x_0 - \x^*\|^2$ for all $t \geq 1$. Therefore, we conclude that 
\begin{eqnarray*}
\EE\left[\sum_{t=0}^T\epsilon(\x_t)\right] & \leq & \frac{\|\x_0 - \x^*\|^2}{2(\lambda - \overline{\eta} - \tau\overline{\eta})\eta_1} + \frac{\|\x_0 - \x^*\|^2}{2(\lambda - \overline{\eta} - \tau\overline{\eta})}\sum_{t=1}^T \left(\frac{1}{\eta_{t+1}} - \frac{1}{\eta_t}\right) \\
& = & \frac{\|\x_0 - \x^*\|^2}{2(\lambda - \overline{\eta} - \tau\overline{\eta})\eta_1} + \frac{\|\x_0 - \x^*\|^2}{2(\lambda - \overline{\eta} - \tau\overline{\eta})\eta_{T+1}} \\ 
& \leq & \frac{\|\x_0 - \x^*\|^2}{(\lambda - \overline{\eta} - \tau\overline{\eta})\underline{\eta}} \ = \ O(1). 
\end{eqnarray*}
This completes the proof. 
\end{proof}
Inspired by Lemma~\ref{Lemma:OGD-noisy-relative-iterative}, we impose an intuitive condition on the variance ratio of noisy process $\{\tau_t\}_{t \geq 0}$. More specifically, $\{\tau_t\}_{t \geq 0}$ is an nonincreasing sequence and there exists a function $a: \br^+ \rightarrow \br^+$ satisfying $a(t) = o(1)$ such that 
\begin{equation}\label{condition:noisy-OGD-relative}
\frac{1}{T+1}\left(\sum_{t=0}^{T-1} \tau_t\right) \ = \ O(a(T)). 
\end{equation}
\begin{remark}
The condition~\eqref{condition:noisy-OGD-relative} is fairly mild. Indeed, the decaying rate $a(t)$ can be very slow which still guarantees the finite-time last-iterate convergence rate. For some typical examples, we have $a(t) = \log\log(t)/t$ if $\tau_t = 1/t\log(t)$ and $a(t) = \log(t)/t$ if $\tau_t = 1/t$. When $\tau_t = \Omega(1/t)$, we have $a(t) = \tau_t$, such as $a(t) = 1/\sqrt{t}$ if $\tau_t = 1/\sqrt{t}$ and $a(t) = 1/\log\log(t)$ if $\tau_t = 1/\log\log(t)$. Under this condition, we can derive the last-iterate convergence rate given the decaying rate of $\tau_t$ as $t \rightarrow +\infty$.  
\end{remark}
Under the condition~\eqref{condition:noisy-OGD-relative}, the finite-time last-iterate convergence rate can be derived under certain step-size sequences. 
\begin{theorem}\label{Thm:noisy-OGD-relative-last-rate}
Fix a $\lambda$-cocoercive game $\GCal$ with continuous action spaces $(\NCal, \XCal = \prod_{i=1}^N \br^{n_i}, \{u_i\}_{i=1}^N)$ with an nonempty set of Nash equilibrium $\XCal^*$. Under the noisy model~\eqref{model:noisy} with relative random noise~\eqref{model:noisy-relative} satisfying Eq.~\eqref{condition:noisy-OGD-relative} and the step-size sequence satisfying $0 < \underline{\eta} \leq \eta_t \leq \overline{\eta} < \lambda/(1+\tau)$ for all $t \geq 1$, the noisy OGD iterate $\x_t$ satisfies
\begin{equation*}
\EE[\epsilon(\x_T)] \ = \ \left\{ 
\begin{array}{rl}
O(a(T)) & \textnormal{if} \ a(T) = \Omega(1/T), \\
O(1/T) & \textnormal{otherwise}.
\end{array}
\right. 
\end{equation*}
\end{theorem}
\begin{proof}
Using Lemma~\ref{Lemma:OGD-noisy-relative-iterative} and $\eta_t \geq \underline{\eta} > 0$, we have
\begin{equation*}
\EE[\epsilon(\x_T)] \ \leq \ \EE[\epsilon(\x_t)] + \frac{\sum_{j=t}^{T-1} \tau_j\|\sv(\x_j)\|^2}{\lambda\underline{\eta}}. 
\end{equation*}
Summing up the above inequality over $t = 0, \ldots, T$ yields
\begin{equation*}
(T+1)\EE[\epsilon(\x_T)] \ \leq \ \sum_{t=0}^T \EE[\epsilon(\x_t)] + \frac{\sum_{t=0}^{T-1} \sum_{j=t}^{T-1} \tau_j\|\sv(\x_j)\|^2}{\lambda\underline{\eta}}. 
\end{equation*}
Using Eq.~\eqref{inequality-noisy-OGD-relative-first} and $\eta_t \geq \underline{\eta} > 0$ for all $t \geq 1$, we have
\begin{equation*}
\EE[\|\sv(\x_j)\|^2] \ \leq \ \frac{\EE[\|\x_j - \x^*\|^2] - \EE[\|\x_{j+1} - \x^*\|^2]}{2(\lambda - \overline{\eta} - \tau\overline{\eta})\underline{\eta}}. 
\end{equation*}
Putting these pieces with the fact that $\{\tau_t\}_{t \geq 0}$ is an nonincreasing sequence yields that 
\begin{equation*}
\sum_{t=0}^{T-1} \sum_{j=t}^{T-1} \tau_j\|\sv(\x_j)\|^2 \ \leq \ \left(\sum_{t=0}^{T-1} \tau_t\right)\left(\sum_{t=0}^{T-1} \|\sv(\x_t)\|^2\right) \ \leq \ \frac{\|\x_0 - \x^*\|^2}{2(\lambda - \overline{\eta} - \tau\overline{\eta})\underline{\eta}}\left(\sum_{t=0}^{T-1} \tau_t\right). 
\end{equation*}
Together with the fact that $\EE[\sum_{t=0}^T\epsilon(\x_t)] = O(1)$ (cf. Theorem~\ref{Thm:noisy-OGD-relative-average-rate}), we have
\begin{equation*}
\EE[\epsilon(\x_T)] \ \leq \ \frac{\sum_{t=0}^T \EE[\epsilon(\x_t)]}{T+1} + \frac{\|\x_0 - \x^*\|^2}{T+1}\left(\sum_{t=0}^{T-1} \tau_t\right). 
\end{equation*}
This completes the proof. 
\end{proof}
\begin{algorithm}[!t]
\caption{Adaptive Online Gradient Descent with Noisy Feedback Information}\label{algorithm:AdaOGD_noisy}
\begin{algorithmic}[1]
\STATE \textbf{Initialization:} $\x_0 \in \br^n$, $\eta_1 = 1/\beta$ for some $\beta > 0$ and $r \in (1, +\infty)$.  
\FOR{$t = 0, 1, 2, \ldots$}
\FOR{$i = 1, 2, \ldots, N$}
\STATE $x_i^{t+1} = x_i + \eta_{t+1}\sv(\x_t)$.
\ENDFOR 
\STATE $\Delta\x_{t+1} = \sum_{j=0}^t \eta_{j+1}^{-2}\|\x_j - \x_{j+1}\|^2$
\STATE $\eta_{t+2} = 1/\sqrt{\beta + \log(t+2) + \Delta\x_{t+1}}$. 
\ENDFOR
\end{algorithmic}
\end{algorithm}
\subsection{Adaptive OGD Learning}
We study the convergence property of Algorithm~\ref{algorithm:AdaOGD_noisy} under the noisy model~\eqref{model:noisy} with relative random noise~\eqref{model:noisy-relative} satisfying that there exists $a(t) = o(1)$ such that
\begin{equation}\label{condition:noisy-OGD-relative-adaptive}
\frac{\log(T+1)}{T+1}\left(\sum_{t=0}^{T-1} \tau_t\right) \ = \ O(a(T)). 
\end{equation}
Note that Eq.~\eqref{condition:noisy-OGD-relative-adaptive} is slightly stronger than Eq.~\eqref{condition:noisy-OGD-relative}. 

We present our main result in Theorem~\ref{Thm:OGD-adaptive-relative} and remark that the proof technique is new and can be interpreted as a novel combination of that in Theorem~\ref{Thm:OGD-adaptive} and~\ref{Thm:noisy-OGD-relative-last-rate}.  
\begin{theorem}\label{Thm:OGD-adaptive-relative}
Fix a $\lambda$-cocoercive game $\GCal$ with continuous action spaces $(\NCal, \XCal = \prod_{i=1}^N \br^{n_i}, \{u_i\}_{i=1}^N)$ with an nonempty set of Nash equilibrium $\XCal^*$. Under the noisy model~\eqref{model:noisy} with relative random noise~\eqref{model:noisy-relative} satisfying Eq.~\eqref{condition:noisy-OGD-relative-adaptive}, the adaptive noisy OGD iterate $\x_t$ satisfies 
\begin{equation*}
\EE[\epsilon(\x_T)] \ = \ \left\{ 
\begin{array}{cl}
O(a(T)) & \textnormal{if} \ a(T) = \Omega(\log(T)/T), \\
O(\log(T)/T) & \textnormal{otherwise}. 
\end{array}
\right. 
\end{equation*}
\end{theorem}
\begin{proof}
Since the step-size sequence $\{\eta_t\}_{t \geq 1}$ is decreasing and converges to zero, we define the first iconic time in our analysis as follows, 
\begin{equation*}
t^* \ = \ \max\left\{t \geq 0 \mid \eta_{t+1} > \frac{\lambda}{2(1+\tau)}\right\} \ < \ +\infty.  
\end{equation*}
First, we claim that $\EE[\|\x_t - \Pi_{\XCal^*}(\x_{t^*})\|] \leq D$ where $D = \max_{1 \leq t \leq t^*} \EE[\|\x_t - \Pi_{\XCal^*}(\x_{t^*})\|]$. Indeed, it suffices to show that $\EE[\|\x_t - \Pi_{\XCal^*}(\x_{t^*})\|] \leq \EE[\|\x_{t^*} - \Pi_{\XCal^*}(\x_{t^*})\|]$ holds for $t > t^*$. By the definition of $t^*$, we have $\eta_{t+1} < \lambda/(1+\tau)$ for all $t > t^*$. The desired inequality follows from Eq.~\eqref{inequality-noisy-OGD-relative-second} and the fact that $\EE[\epsilon(\x_t)] \geq 0$ for all $t > t^*$.  

Furthermore, we derive an upper bound for the term $\sum_{t=0}^T \|\sv(\x_t)\|^2$. Using the update formula (cf. Eq.~\eqref{Update:OGD-learning}) to obtain that
\begin{equation*}
(\x_{t+1} - \x^*)^\top(\sv(\x_t) + \xi_{t+1}) = \frac{1}{2\eta_{t+1}}\left(\|\x_t - \x_{t+1}\|^2 + \|\x^* - \x_{t+1}\|^2 - \|\x^* - \x_t\|^2\right).
\end{equation*}
Recall that $\GCal$ is $\lambda$-cocoercive and the noisy model is defined with relative random noise, we have
\begin{equation*}
\EE[(\x_t - \x^*)^\top(\sv(\x_t) + \xi_{t+1}) \mid \FCal_t] \ \geq \ \lambda\|\sv(\x_t)\|^2. 
\end{equation*}
Using Young's inequality, we have
\begin{equation*}
\EE[(\x_{t+1} - \x_t)^\top(\sv(\x_t) + \xi_{t+1}) \mid \FCal_t] \ \geq \ -\frac{\lambda\|\sv(\x_t)\|^2}{2} - \frac{(1+\tau)\EE[\|\x_{t+1} - \x_t\|^2 \mid \FCal_t]}{\lambda}.  
\end{equation*}
Putting these pieces together and taking an expectation yields that 
\begin{equation}\label{inequality-noisy-OGD-adaptive-first}
\lambda\EE[\|\sv(\x_t)\|^2] \ \leq \ \EE\left[\frac{\|\x^* - \x_t\|^2 - \|\x^* - \x_{t+1}\|^2}{\eta_{t+1}}\right] + \EE\left[\left(\frac{2(1+\tau)}{\lambda} - \frac{1}{\eta_{t+1}}\right)\|\x_t - \x_{t+1}\|^2\right].
\end{equation}
Recall that the step-size sequence $\{\eta_t\}_{t \geq 1}$ is nonincreasing and $\|\x_t - \Pi_{\XCal^*}(\x_{t^*})\| \leq D$, we let $\x^* = \Pi_{\XCal^*}(\x_{t^*})$ in Eq.~\eqref{inequality-noisy-OGD-adaptive-first} and obtain that 
\begin{equation*}
\sum_{t=0}^T \lambda\EE[\|\sv(\x_t)\|^2] \ \leq \ \EE\left[\frac{D^2}{\eta_{T+1}}\right] + \sum_{t=0}^T \EE\left[\left(\frac{2(1+\tau)}{\lambda} - \frac{1}{\eta_{t+1}}\right)\|\x_t - \x_{t+1}\|^2\right]. 
\end{equation*}
To proceed, we define the second iconic time as 
\begin{equation*}
t_1^* \ = \ \max\left\{t \geq 0 \mid \eta_{t+1} > \frac{\lambda}{4(1+\tau)D^2 + 2(1+\tau)}\right\} \ > \ t^*.  
\end{equation*}
It is clear that $t_1^* < +\infty$ and $\eta_{t+1} \leq \lambda/(2+2\tau)$ for all $t > t_1^*$ which implies that $(2+2\tau)/\lambda - 1/\eta_{t+1} \leq 0$. Assume $T$ sufficiently large without loss of generality, we have
\begin{equation*}
\sum_{t=0}^T \lambda\EE[\|\sv(\x_t)\|^2] \ \leq \ \EE\left[\frac{D^2}{\eta_{T+1}}\right] + \frac{2(1+\tau)}{\lambda}\left(\sum_{t=0}^{t_1^*} \EE[\|\x_t - \x_{t+1}\|^2]\right) \ = \ \text{I} + \text{II}. 
\end{equation*}
We also use Lemmas~A.1 and~A.2 from \citet{Bach-2019-Universal} to bound terms $\text{I}$ and $\text{II}$. For convenience, we present these two lemmas here: 
\begin{lemma}\label{Lemma:noisy-OGD-adaptive-first}
For a sequence of numbers $a_0, a_1, \ldots, a_n \in [0, a]$ and $b \geq 0$, the following inequality holds: 
\begin{equation*}
\sqrt{b + \sum_{i=0}^{n-1} a_i} - \sqrt{b} \ \leq \ \sum_{i=0}^n \frac{a_i}{\sqrt{b + \sum_{j=0}^{i-1} a_j}} \ \leq \ \frac{2a}{\sqrt{b}} + 3\sqrt{a} + 3\sqrt{b + \sum_{i=0}^{n-1} a_i}. 
\end{equation*}
\end{lemma}
\paragraph{Bounding $\text{I}$:} We derive from the definition of $\eta_t$ and Jensen's inequality that 
\begin{equation*}
\text{I} \ \leq \ D^2\sqrt{\beta + \log(T+1) + \sum_{j=0}^{T-1} \EE\left[\frac{\|\x_j - \x_{j+1}\|^2}{\eta_{j+1}^2}\right]}. 
\end{equation*}
Since $\EE[\|\x_t - \Pi_{\XCal^*}(\x_{t^*})\|] \leq D$ for all $t \geq 0$ and the notion of $\lambda$-cocercivity implies the notion of $(1/\lambda)$-Lipschiz continuity, we have
\begin{equation*}
\EE\left[\frac{\|\x_t - \x_{t+1}\|^2}{\eta_{t+1}^2}\right] \ \leq \ (2+2\tau)\EE[\|\sv(\x_t)\|^2] \ \leq \ \frac{(2+2\tau)\|\x_t - \Pi_{\XCal^*}(\x_{t^*})\|^2}{\lambda^2} \ \leq \ \frac{(2+2\tau)D^2}{\lambda^2}. 
\end{equation*}
Using the first inequality in Lemma~\ref{Lemma:noisy-OGD-adaptive-first}, we have
\begin{eqnarray}\label{inequality-noisy-OGD-adaptive-second}
\text{I} & \leq & D^2\sqrt{\beta + \log(T+1)} + \sum_{t=0}^T \frac{D^2\EE[\|\x_t - \x_{t+1}\|^2/\eta_{t+1}^2]}{\sqrt{\beta + \log(T+1) + \sum_{j=0}^{t-1} \EE[\|\x_j - \x_{j+1}\|^2/\eta_{j+1}^2]}} \nonumber \\
& \leq & D^2\sqrt{\beta + \log(T+1)} + \sum_{t=0}^{t_1^*} \frac{D^2\EE[\|\x_t - \x_{t+1}\|^2/\eta_{t+1}^2]}{\sqrt{\beta + \log(T+1) + \sum_{j=0}^{t-1} \EE[\|\x_j - \x_{j+1}\|^2/\eta_{j+1}^2]}} \nonumber \\ 
& & + \sum_{t=t_1^*+1}^T D^2\eta_{t+1} \EE\left[\frac{\|\x_t - \x_{t+1}\|^2}{\eta_{t+1}^2}\right] \nonumber \\ 
& \leq & D^2\sqrt{\beta + \log(T+1)} + \sum_{t=0}^{t_1^*} \frac{D^2\EE[\|\x_t - \x_{t+1}\|^2/\eta_{t+1}^2]}{\sqrt{\beta + \log(T+1) + \sum_{j=0}^{t-1} \EE[\|\x_j - \x_{j+1}\|^2/\eta_{j+1}^2]}} \nonumber \\
& & + \sum_{t=t_1^*+1}^T (2+2\tau)D^2\eta_{t+1}\EE[\|\sv(\x_t)\|^2].  
\end{eqnarray}
Since $\eta_{t+1} \leq \lambda/[4(1+\tau)D^2]$ for all $t > t_1^*$, we have
\begin{equation}\label{inequality-noisy-OGD-adaptive-third}
\sum_{t=t_1^*+1}^T (2+2\tau)D^2\eta_{t+1}\EE[\|\sv(\x_t)\|^2] \ \leq \ \sum_{t=t_1^*+1}^T \frac{\lambda\EE[\|\sv(\x_t)\|^2]}{2}.
\end{equation}
Using the second inequality in Lemma~\ref{Lemma:OGD-adaptive-first}, we have
\begin{eqnarray}\label{inequality-noisy-OGD-adaptive-fourth}
& & \sum_{t=0}^{t_1^*} \frac{D^2\EE[\|\x_t - \x_{t+1}\|^2/\eta_{t+1}^2]}{\sqrt{\beta + \log(T+1) + \sum_{j=0}^{t-1} \EE[\|\x_j - \x_{j+1}\|^2/\eta_{j+1}^2]}} \\ 
& \leq & \frac{(4+4\tau)D^2}{\lambda^2\sqrt{\beta + \log(T+1)}} + \frac{3D\sqrt{2+2\tau}}{\lambda} + 3\sqrt{\beta + \log(T+1) + \sum_{j=0}^{t_1^*-1} \EE\left[\frac{\|\x_j - \x_{j+1}\|^2}{\eta_{j+1}^2}\right]}. \nonumber
\end{eqnarray}
By the definition of $\eta_t$, we have
\begin{eqnarray}\label{inequality-noisy-OGD-adaptive-fifth}
& & \sqrt{\beta + \log(T+1) + \sum_{j=0}^{t_1^*-1} \EE\left[\frac{\|\x_j - \x_{j+1}\|^2}{\eta_{j+1}^2}\right]} \\ 
& \leq & \frac{1}{\eta_{t_1^*+1}} + \sqrt{\log(T+1)} \ < \ \frac{4(1+\tau)D^2+2(1+\tau)}{\lambda} + \sqrt{\log(T+1)}. \nonumber
\end{eqnarray}
Putting Eq.~\eqref{inequality-noisy-OGD-adaptive-third}-\eqref{inequality-noisy-OGD-adaptive-fifth} together yields that 
\begin{eqnarray*}
\text{I}  & \leq & D^2\sqrt{\beta + \log(T+1)} + \frac{(4+4\tau)D^2}{\lambda^2\sqrt{\beta + \log(T+1)}} + \frac{3D\sqrt{2+2\tau}}{\lambda} + \frac{12(1+\tau)D^2+6(1+\tau)}{\lambda} \\
& & + \sqrt{\log(T+1)} + \sum_{t=t_1^*+1}^T \frac{\lambda\EE[\|\sv(\x_t)\|^2]}{2}. 
\end{eqnarray*}
\paragraph{Bounding $\text{II}$:} Recalling that 
\begin{equation*}
\EE\left[\frac{\|\x_t - \x_{t+1}\|^2}{\eta_{t+1}^2}\right] \ \leq \ \frac{(2+2\tau)D^2}{\lambda^2},  
\end{equation*}
and $\eta_t \leq 1/\beta$ for all $t \geq 1$, we have
\begin{equation*}
\EE\left[\|\x_t - \x_{t+1}\|^2\right] \ \leq \ \frac{(2+2\tau)D^2}{\lambda^2 \beta^2},  
\end{equation*}
Putting these pieces together yields that 
\begin{equation*}
\text{II} \ = \ \frac{2(1+\tau)}{\lambda}\left(\sum_{t=0}^{t_1^*} \EE[\|\x_t - \x_{t+1}\|^2]\right) \ \leq \ \frac{4(1+\tau)^2D^2 t_1^*}{\lambda^3\beta^2}. 
\end{equation*}
Therefore, we have
\begin{eqnarray*}
\sum_{t=0}^T \frac{\lambda\EE[\|\sv(\x_t)\|^2]}{2} & \leq & D^2\sqrt{\beta + \log(T+1)} + \sqrt{\log(T+1)} + \frac{(4+4\tau)D^2}{\lambda^2\sqrt{\beta + \log(T+1)}} \\
& & + \frac{3D\sqrt{2+2\tau}}{\lambda} + \frac{12(1+\tau)D^2+6(1+\tau)}{\lambda} + \frac{4(1+\tau)^2D^2 t_1^*}{\lambda^3 \beta^2}. 
\end{eqnarray*}
By the definition, we have $t_1^* < +\infty$ is uniformly bounded. To this end, we conclude that $\sum_{t=0}^T \EE[\|\sv(\x_t)\|^2] \leq C_1 + C_2\sqrt{\log(T+1)}$, where $C_1 > 0$ and $C_2 > 0$ are universal constants. 

Finally, we proceed to bound the term $\epsilon(\x_T)$. Without loss of generality, we can start the sequence at a later index $t_1^*$ since $t_1^* < +\infty$. This implies that $\eta_{t+1} \leq \lambda/2(1+\tau)$. Using the last equation in the proof of  Lemma~\ref{Lemma:OGD-noisy-relative-iterative}, we have
\begin{equation*}
\EE[\epsilon(\x_T)] \ \leq \ \EE[\epsilon(\x_t)] + \sum_{j=t}^{T-1} \frac{\tau_j}{\lambda}\EE\left[\frac{\|\sv(\x_j)\|^2}{\eta_{j+1}}\right]. 
\end{equation*}
Summing up the above inequality over $t = t_1^*, \ldots, T+$ yields
\begin{equation*}
(T-t_1^*+1)\EE[\epsilon(\x_T)] \ \leq \ \sum_{t=t_1^*}^T \EE[\epsilon(\x_t)] + \frac{1}{\lambda}\left(\sum_{t=t_1^*}^{T-1} \sum_{j=t}^{T-1} \tau_j\EE\left[\frac{\|\sv(\x_j)\|^2}{\eta_{j+1}}\right]\right). 
\end{equation*}
Since $\{\tau_t\}_{t \geq 0}$ is an nonincreasing sequence, we have
\begin{equation*}
\sum_{t=t_1^*}^{T-1} \sum_{j=t}^{T-1} \tau_j\EE\left[\frac{\|\sv(\x_j)\|^2}{\eta_{j+1}}\right] \ \leq \ \left(\sum_{t=0}^{T-1} \tau_j\right)\left(\sum_{t=t_1^*}^{T-1} \EE\left[\frac{\|\sv(\x_j)\|^2}{\eta_{j+1}}\right]\right). 
\end{equation*}
Using Eq.~\eqref{inequality-noisy-OGD-relative-first} and $\eta_{t+1} \leq \lambda/2(1+\tau)$ for all $t > t_1^*$, we have
\begin{equation*}
\EE\left[\frac{\|\sv(\x_j)\|^2}{\eta_{j+1}}\right] \ \leq \ \EE\left[\frac{\|\x_j - \x^*\|^2 - \|\x_{j+1} - \x^*\|^2}{\eta_{j+1}^2}\right]. 
\end{equation*}
Note that $\{\eta_t\}_{t \geq 0}$ is an nonnegative and nonincreasing sequence and $\EE[\|\x_{j+1} - \x^*\|^2] \leq D^2$. Putting these pieces together yields that 
\begin{eqnarray*}
\sum_{t=t_1^*}^{T-1} \EE\left[\frac{\|\sv(\x_j)\|^2}{\eta_{j+1}}\right] & \leq & \EE\left[\frac{D^2}{\eta_T^2}\right] \ \leq \ D^2\left(\beta + \log(T) + \sum_{t=0}^T \EE\left[\frac{\|\x_t - \x_{t+1}\|^2}{\eta_{t+1}^2}\right]\right) \\
& \leq & D^2\left(\beta + \log(T) + 2(1+\tau)\sum_{t=0}^T \EE\left[\|\sv(\x_t)\|^2\right]\right) \\
& = & O(\log(T)). 
\end{eqnarray*}
Therefore, we conclude that 
\begin{equation*}
\EE[\epsilon(\x_T)] \ \leq \ \frac{\sum_{t=t_1^*}^T \EE[\epsilon(\x_t)]}{T-t_1^*+1} + \frac{C\log(T+1)}{\lambda(T-t_1^*+1)}\left(\sum_{t=0}^{T-1} \tau_t\right) \text{ for some } C > 0.  
\end{equation*}
This completes the proof. 
\end{proof}

\section*{Acknowledgments}
We would like to thank three anonymous referees for constructive suggestions that improve the quality of this paper. Zhengyuan Zhou was supported by the IBM Goldstine Fellowship. This work was supported in part by the Mathematical Data Science program of the Office of Naval Research under grant number N00014-18-1-2764.

\bibliographystyle{plainnat}
\bibliography{ref}

\appendix \onecolumn
\section{Convergence under Imperfect Feedback with Absolute Random Noise}\label{sec:imperfect_absolute}
In this section, we analyze the convergence of OGD-based learning on $\lambda$-cocercive games under imperfect feedback with absolute random noise~\eqref{model:noisy-absolute}. We first establish that OGD under noisy feedback converges almost surely in last-iterate to the set of Nash equilibria of a co-coercive game if $\sigma_t^2 \in (0, \sigma^2)$ for some $\sigma^2 < +\infty$ and the finite-time $O(1/\sqrt{T})$ convergence rate on $(1/T)\EE[\sum_{t=0}^T \epsilon(\x_t)]$ under properly diminishing step-size sequences. We also present a finite-time convergence rate on $\EE[\epsilon(\x_T)]$ if $\sigma_t^2$ satisfies certain conditions. 

\subsection{Almost Sure Last-Iterate Convergence}
We start by developing a key iterative formula for $\EE[\epsilon(\x_t)]$ in the following lemma. 
\begin{lemma}\label{Lemma:OGD-noisy-absolute-iterative}
Fix a $\lambda$-cocoercive game $\GCal$ with a continuous action space, $\GCal = (\NCal, \XCal = \prod_{i=1}^N \br^{n_i}, \{u_i\}_{i=1}^N)$, and let the set of Nash equilibria, $\XCal^*$, be nonempty. Under the noisy model~\eqref{model:noisy} with absolute random noise~\eqref{model:noisy-absolute} and letting the OGD-based learning run with a step-size sequence $\eta_t \in (0, \lambda)$, the noisy OGD iterate $\x_t$ satisfies
\begin{equation*}
\EE[\epsilon(\x_{t+1})] \ \leq \ \EE[\epsilon(\x_t)] + \frac{\sigma_t^2}{\lambda\eta_{t+1}}. 
\end{equation*}
\end{lemma}

We are now ready to establish last-iterate convergence in a strong, almost sure sense. Note that the conditions imposed on $\sigma_t^2$ and $\eta_t$ are minimal. 
\begin{theorem}\label{Thm:noisy-OGD-absolute-last}
Fix a $\lambda$-cocoercive game $\GCal$ with a continuous action space, $\GCal = (\NCal, \XCal = \prod_{i=1}^N \br^{n_i}, \{u_i\}_{i=1}^N)$, and let the set of Nash equilibria, $\XCal^*$, be nonempty. Consider the noisy model~\eqref{model:noisy} with absolute random noise~\eqref{model:noisy-absolute} satisfying $\sigma_t^2 \in (0, \sigma^2]$ for some $\sigma^2 < +\infty$ and letting the OGD-based learning run with a step-size sequence satisfying
\begin{equation*}
\sum_{t=1}^\infty \eta_t \ = \ + \infty, \qquad \sum_{t=1}^\infty \eta_t^2 \ < \ + \infty.  
\end{equation*}
Then the noisy OGD iterate $\x_t$ converges to $\XCal^*$ almost surely. 
\end{theorem}

\subsection{Finite-Time Convergence Rate: Time-Average and Last-Iterate}
For completeness, we characterize two types of rates: the time-average and last-iterate convergence rate, as formalized by the following theorems. 
\begin{theorem}\label{Thm:noisy-OGD-absolute-average-rate}
Fix a $\lambda$-cocoercive game $\GCal$ with a continuous action space, $\GCal = (\NCal, \XCal = \prod_{i=1}^N \br^{n_i}, \{u_i\}_{i=1}^N)$, and let the set of Nash equilibria, $\XCal^*$, be nonempty. Under the noisy model~\eqref{model:noisy} with absolute random noise~\eqref{model:noisy-absolute} satisfying $\sigma_t^2 \in (0, \sigma^2]$ for some $\sigma^2 < +\infty$ and letting the OGD-based learning run with a step-size sequence $\eta_t = c/\sqrt{t}$ for some constant $c \in (0, \lambda)$, the noisy OGD iterate $\x_t$ satisfies
\begin{equation*}
\frac{1}{T+1}\left(\EE\left[\sum_{t=0}^T \epsilon(\x_t)\right]\right) \ = \ O\left(\frac{\log(T)}{\sqrt{T}}\right). 
\end{equation*}
\end{theorem}
Inspired by Lemma~\ref{Lemma:OGD-noisy-absolute-iterative}, we impose an intuitive condition on the variance of noisy process $\{\sigma_t^2\}_{t \geq 0}$. More specifically, there exists a function $\alpha: \br^+ \rightarrow \br^+$ satisfying $a(t) = o(1)$ and $a(t) = \Omega(1/t)$ such that 
\begin{equation}\label{condition:noisy-OGD-absolute}
\frac{1}{T+1}\left(\sum_{t=0}^{T-1} (t+1)\sigma_t^2\right) \ = \ O(a(T)). 
\end{equation}
Under this condition, the noisy iterate generated by the OGD-based learning achieves the finite-time last-iterate convergence rate regardless of a sequence of possibly constant step-sizes $\eta_t$ satisfying $0 < \underline{\eta} \leq \eta_t \leq \overline{\eta} < \lambda$ for all $t \geq 1$.   
\begin{theorem}\label{Thm:noisy-OGD-absolute-last-rate}
Fix a $\lambda$-cocoercive game $\GCal$ with a continuous action space, $\GCal = (\NCal, \XCal = \prod_{i=1}^N \br^{n_i}, \{u_i\}_{i=1}^N)$, and let the set of Nash equilibria, $\XCal^*$, be nonempty. Under the noisy model~\eqref{model:noisy} with absolute random noise~\eqref{model:noisy-absolute} satisfying Eq.~\eqref{condition:noisy-OGD-absolute} and letting the OGD-based learning run with an nonincreasing step-size sequence satisfying $0 < \underline{\eta} \leq \eta_t \leq \overline{\eta} < \lambda$ for all $t \geq 1$, the noisy OGD iterate $\x_t$ satisfies
\begin{equation*}
\EE[\epsilon(\x_T)] \ = \ O(a(T)). 
\end{equation*}
\end{theorem}

\section{Proof of Lemma~\ref{Lemma:OGD-nonadaptive}}
Since $\XCal_i = \br^{n_i}$, we have
\begin{eqnarray*}
& & \|x_{i, t+2} - x_{i, t+1}\|^2 \\ 
& = & \|x_{i, t+1} - x_{i, t} + \eta(v_i(\x_{t+1}) - v_i(\x_t))\|^2 \\
& = & \|x_{i, t+1} - x_{i, t}\|^2 + 2\eta(x_{i, t+1} - x_{i, t})^\top(v_i(\x_{t+1}) - v_i(\x_t)) + \eta^2\|v_i(\x_{t+1}) - v_i(\x_t)\|^2.   
\end{eqnarray*}
Expanding the right-hand side of the above inequality and summing up the resulting inequality over $i \in \NCal$ yields that 
\begin{eqnarray}\label{inequality-OGD-nonadaptive-first}
& & \|\x_{t+2} - \x_{t+1}\|^2 \ = \ \sum_{i \in \NCal} \|x_{i, t+2} - x_{i, t+1}\|^2 \\
& \leq & \sum_{i \in \NCal} \left(\|x_{i, t+1} - x_{i, t}\|^2 + \eta^2\|v_i(\x_{t+1}) - v_i(\x_t)\|^2 + 2\eta(x_{i, t+1} - x_{i, t})^\top(v_i(\x_{t+1}) - v_i(\x_t))\right) \nonumber \\
& = & \|\x_{t+1} - \x_t\|^2 + 2\eta(\x_{t+1} - \x_t)^\top(\sv(\x_{t+1}) - \sv(\x_t)) + \eta^2\|\sv(\x_{t+1}) - \sv(\x_t)\|^2. \nonumber
\end{eqnarray}
Since $\GCal$ is a $\lambda$-cocoercive game, we have
\begin{equation*}
(\x_{t+1} - \x_t)^\top(\sv(\x_{t+1}) - \sv(\x_t)) \leq -\lambda\|\sv(\x_{t+1}) - \sv(\x_t)\|^2. 
\end{equation*}
Plugging the above equation into Eq.~\eqref{inequality-OGD-nonadaptive-first} together with the condition $\eta \in (0, \lambda]$ yields that 
\begin{equation*}
\|\x_{t+2} - \x_{t+1}\|^2 \ \leq \ \|\x_{t+1} - \x_t\|^2. 
\end{equation*}
Using the update formula in Eq.~\eqref{Update:OGD-learning}, we have $\|\sv(\x_{t+1})\| \leq \|\sv(\x_t)\|$ for all $t \geq 0$. 

Then we proceed to bound $\sum_{t=0}^{+\infty} \|\sv(\x_t)\|^2$. Indeed, for any $x_i \in \XCal_i$, we have
\begin{equation*}
(x_i - x_{i, t+1})^\top(x_{i, t+1} - x_{i, t} - \eta v_i(\x_t)) \ = \ 0. 
\end{equation*}
Applying the equality $a^\top b = (\|a+b\|^2 - \|a\|^2 - \|b\|^2)/2$ yields that
\begin{equation*}
(x_{i, t+1} - x_i)^\top v_i(\x_t) \ = \ \frac{1}{2\eta}\left(\|x_{i, t} - x_{i, t+1}\|^2 + \|x_i - x_{i, t+1}\|^2 - \|x_i - x_{i, t}\|^2\right). 
\end{equation*}
Summing up the resulting inequality over $i \in \NCal$ yields that 
\begin{equation*}
(\x_{t+1} - \x)^\top \sv(\x_t) \ = \ \frac{1}{2\eta}\left(\|\x_t - \x_{t+1}\|^2 + \|\x - \x_{t+1}\|^2 - \|\x - \x_t\|^2\right), \quad \forall \x \in \XCal. 
\end{equation*}
Letting $\x = \x^* \in \XCal^*$, we have
\begin{equation}\label{inequality-OGD-nonadaptive-second}
(\x_{t+1} - \x^*)^\top\sv(\x_t) \ = \ \frac{1}{2\eta}\left(\|\x_t - \x_{t+1}\|^2 + \|\x^* - \x_{t+1}\|^2 - \|\x^* - \x_t\|^2\right).
\end{equation}
Furthermore, we have
\begin{equation*}
(\x_{t+1} - \x^*)^\top\sv(\x_t) = (\x_t - \x^*)^\top\sv(\x_t) + (\x_{t+1} - \x_t)^\top\sv(\x_t). 
\end{equation*}
Since $\GCal$ is a $\lambda$-cocoercive game and $\sv(\x^*)=0$, we have
\begin{equation*}
(\x_t - \x^*)^\top\sv(\x_t) \ = \ (\x_t - \x^*)^\top(\sv(\x_t) - \sv(\x^*)) \ \leq \ -\lambda\|\sv(\x_t) - \sv(\x^*)\|^2 \ = \ -\lambda\|\sv(\x_t)\|^2. 
\end{equation*}
By Young's inequality we have
\begin{equation*}
(\x_{t+1} - \x_t)^\top\sv(\x_t) \ \leq \ \frac{\lambda\|\sv(\x_t)\|^2}{2} + \frac{\|\x_{t+1} - \x_t\|^2}{2\lambda}. 
\end{equation*}
Putting these pieces together yields that
\begin{equation}\label{inequality-OGD-nonadaptive-third}
(\x_{t+1} - \x^*)^\top\sv(\x_t) \ \leq \ \frac{\|\x_{t+1} - \x_t\|^2}{2\lambda} - \frac{\lambda\|\sv(\x_t)\|^2}{2}. 
\end{equation}
Plugging Eq.~\eqref{inequality-OGD-nonadaptive-third} into Eq.~\eqref{inequality-OGD-nonadaptive-second} together with the condition $\eta \in (0, \lambda]$ yields that 
\begin{equation*}
\lambda\|\sv(\x_t)\|^2 \ \leq \ \frac{\|\x^* - \x_t\|^2 - \|\x^* - \x_{t+1}\|^2}{\eta}. 
\end{equation*}
Summing up the above inequality over $t = 0, 1, 2, \ldots$ and using the boundedness of $\XCal$ yields that
\begin{equation*}
\sum_{t=0}^{+\infty} \|\sv(\x_t)\|^2 \ \leq \ \frac{\|\x^* - \x_0\|^2}{\eta\lambda}. 
\end{equation*}
Note that $\x^* \in \XCal^*$ is chosen arbitrarily, we let $\x^* = \Pi_{\XCal^*}(\x_0)$ and conclude the desired inequality. 

\section{Postponed Proofs in Section~\ref{sec:imperfect_absolute}}
In this section, we present the missing proofs in Section~\ref{sec:imperfect_absolute}. 
\subsection{Proof of Lemma~\ref{Lemma:OGD-noisy-absolute-iterative}}
By the definition of $\epsilon(\x)$, we have
\begin{eqnarray*}
\epsilon(\x_{t+1}) - \epsilon(\x_t) & = & \|\sv(\x_{t+1})\|^2 - \|\sv(\x_t)\|^2 \ = \ (\sv(\x_{t+1}) - \sv(\x_t))^\top(\sv(\x_{t+1}) + \sv(\x_t)) \\
& = & 2(\sv(\x_{t+1}) - \sv(\x_t))^\top\sv(\x_t) + \|\sv(\x_{t+1}) - \sv(\x_t)\|^2. 
\end{eqnarray*}
Using the update formula in Eq.~\eqref{Update:OGD-learning}, it holds that $\sv(\x_t) = \eta_{t+1}^{-1}(\x_{t+1} - \x_t) - \xi_{t+1}$. Therefore, we have
\begin{equation*}
\epsilon(\x_{t+1}) - \epsilon(\x_t) \ = \ \frac{2}{\eta_{t+1}}\left((\sv(\x_{t+1}) - \sv(\x_t))^\top(\x_{t+1} - \x_t) - (\sv(\x_{t+1}) - \sv(\x_t))^\top\xi_{t+1}\right) + \|\sv(\x_{t+1}) - \sv(\x_t)\|^2. 
\end{equation*}
Since $\GCal$ is a $\lambda$-cocoercive game, we have
\begin{equation*}
(\sv(\x_{t+1}) - \sv(\x_t))^\top(\x_{t+1} - \x_t) \ \leq \ -\lambda\|\sv(\x_{t+1}) - \sv(\x_t)\|^2. 
\end{equation*}
Using  Young's inequality, we have
\begin{equation*}
- (\sv(\x_{t+1}) - \sv(\x_t))^\top\xi_{t+1} \ \leq \ \frac{\lambda\|\sv(\x_{t+1}) - \sv(\x_t)\|^2}{2} + \frac{\|\xi_{t+1}\|^2}{2\lambda}. 
\end{equation*}
Putting these pieces together yields that 
\begin{equation}\label{inequality-noisy-OGD-absolute-second}
\epsilon(\x_{t+1}) - \epsilon(\x_t) \ \leq \ \frac{\|\xi_{t+1}\|^2}{\lambda\eta_{t+1}} + \left(1 - \frac{\lambda}{\eta_{t+1}}\right)\|\sv(\x_{t+1}) - \sv(\x_t)\|^2.
\end{equation}
Taking an expectation of Eq.~\eqref{inequality-noisy-OGD-absolute-second} conditioned on $\FCal_t$ yields that
\begin{equation*}
\EE[\epsilon(\x_{t+1}) \mid \FCal_t] - \epsilon(\x_t) \ \leq \ \frac{\EE[\|\xi_{t+1}\|^2 \mid \FCal_t]}{\lambda\eta_{t+1}} + \left(1 - \frac{\lambda}{\eta_{t+1}}\right)\EE\left[\|\sv(\x_{t+1}) - \sv(\x_t)\|^2 \mid \FCal_t\right]. 
\end{equation*}
Since the noisy model~\eqref{model:noisy} is with absolute random noise~\eqref{model:noisy-absolute}, we have $\EE[\|\xi_{t+1}\|^2 \mid \FCal_t] \leq \sigma_t^2$. Also, $\eta_t \in (0, \lambda)$ for all $t \geq 1$. Therefore, we conclude that 
\begin{equation*}
\EE[\epsilon(\x_{t+1}) \mid \FCal_t] - \epsilon(\x_t) \ \leq \ \frac{\sigma_t^2}{\lambda\eta_{t+1}}.  
\end{equation*}
Taking the expectation of both sides yields the desired inequality.

\subsection{Proof of Theorem~\ref{Thm:noisy-OGD-absolute-last}}
Recalling Eq.~\eqref{inequality-noisy-OGD-absolute-key} (cf. Lemma~\ref{Lemma:key-inequality-noisy-OGD}), we take the expectation of both sides conditioned on $\FCal_t$ to obtain
\begin{equation*}
\EE[\|\x_{t+1} - \x^*\|^2 \mid \FCal_t] \ \leq \ \|\x_t - \x^*\|^2 - (2\lambda - 2\eta_{t+1})\eta_{t+1}\|\sv(\x_t)\|^2 + 2\eta_{t+1}^2\EE[\|\xi_{t+1}\|^2 \mid \FCal_t] + 2\eta_{t+1}\EE[(\x_t - \x^*)^\top\xi_{t+1} \mid \FCal_t]. 
\end{equation*}
Since the noisy model~\eqref{model:noisy} is with absolute random noise~\eqref{model:noisy-absolute} satisfying $\sigma_t^2 \in (0, \sigma^2]$ for some $\sigma^2 < +\infty$, we have $\EE[(\x_t - \x^*)^\top\xi_{t+1} \mid \FCal_t] = 0$ and $\EE[\|\xi_{t+1}\|^2 \mid \FCal_t] \leq \sigma^2$. Therefore, we have
\begin{equation*}
\EE[\|\x_{t+1} - \x^*\|^2 \mid \FCal_t] \ \leq \ \|\x_t - \x^*\|^2 - (2\lambda - 2\eta_{t+1})\eta_{t+1}\|\sv(\x_t)\|^2 + 2\eta_{t+1}^2\sigma^2. 
\end{equation*}
Since $\sum_{t=1}^\infty \eta_t^2 < \infty$, we have $\eta_t \rightarrow 0$ as $t \rightarrow +\infty$. Without loss of generality, we assume $\eta_t \leq \lambda$ for all $t$. Then we have
\begin{equation}\label{inequality-noisy-OGD-absolute-third}
\EE[\|\x_{t+1} - \x^*\|^2 \mid \FCal_t] \ \leq \ \|\x_t - \x^*\|^2 + 2\eta_{t+1}^2\sigma^2. 
\end{equation}
We let $M_t = \|\x_t - \x^*\|^2 + 2\sigma^2(\sum_{j > t} \eta_j^2)$ and obtain that $M_t$ is an nonnegative supermartingale. Then Doob's martingale convergence theorem shows that $M_n$ converges to an nonnegative and integrable random variable almost surely. Let $M_\infty = \lim_{t \rightarrow +\infty} M_t$, it suffices to show that $M_\infty = 0$ almost surely.
	
We first claim that \textit{every neighborhood $U$ of $\XCal^*$ is recurrent}: there exists a subsequence $\x_{t_k}$ of $\x_t$ such that $\x_{t_k} \rightarrow \XCal^*$ almost surely. Equivalently, there exists a Nash equilibria $\x^* \in \XCal^*$ such that $\|\x_{t_k} - \x^*\|^2 \rightarrow 0$ almost surely. To this end, we can define $M_t$ with such Nash equilibria. Since $\sum_{t=1}^\infty \eta_t^2 < \infty$, we have $\sum_{j > t} \eta_j^2 \rightarrow 0$ as $t \rightarrow +\infty$ and the following statement holds almost surely:
\begin{equation*}
\lim_{k \rightarrow +\infty} M_{t_k} \ = \ \lim_{k \rightarrow +\infty} \|\x_{t_k} - \x^*\|^2 \ = \ 0. 
\end{equation*}
Since the whole sequence converges to $M_\infty$ almost surely, we conclude that $M_\infty = 0$ almost surely. 
	
\paragraph{Proof of the recurrence claim:} Let $U$ be a neighborhood of $\XCal^*$ and assume to the contrary that, $\x_t \notin U$ for sufficiently large $t$ with positive probability. By starting the sequence at a later index if necessary and noting that $\sum_{t=1}^\infty \eta_t^2 < \infty$, we may assume that $\x_t \notin U$ and $\eta_t \leq \lambda/2$ for all $t$ without loss of generality. Thus, there exists some $c > 0$ such that $\|\sv(\x_t)\|^2 \geq c$ for all $t$. As a result, for all $\x^* \in \XCal^*$, we let $\psi_{t+1} = (\x_t - \x^*)^\top\xi_{t+1}$ and have
\begin{equation*}
\|\x_{t+1} - \x^*\|^2 \ \leq \ \|\x_t - \x^*\|^2 - \lambda c\eta_{t+1} + 2\eta_{t+1}\psi_{t+1} + 2\eta_{t+1}^2\|\xi_{t+1}\|^2.
\end{equation*}
Summing up the above inequality over $t = 0, 1, \ldots, T$ together with $\theta_t = \sum_{j=1}^t \eta_j$ yields that 
\begin{equation}\label{inequality-noisy-OGD-absolute-last-main}
\|\x_{T+1} - \x^*\|^2 \ \leq \ \|\x_0 - \x^*\|^2 - \lambda c\theta_{T+1} + 2\theta_{T+1}\left[\frac{\sum_{t=1}^{T+1} \eta_t \psi_t}{\theta_{T+1}} + \frac{\sum_{t=1}^{T+1} \eta_t^2\|\xi_t\|^2}{\theta_{T+1}}\right].
\end{equation}
Since the noisy model~\eqref{model:noisy} is with absolute random noise~\eqref{model:noisy-absolute} satisfying $\sigma_t^2 \in (0, \sigma^2]$ for some $\sigma^2 < +\infty$, we have $\EE[\psi_{t+1} \mid \FCal_t] = 0$. Furthermore, we obtain by taking the expectation of both sides of Eq.~\eqref{inequality-noisy-OGD-absolute-third} that 
\begin{equation*}
\EE[\|\x_{t+1} - \x^*\|^2] \ \leq \ \EE[\|\x_t - \x^*\|^2] + 2\eta_{t+1}^2\sigma^2,  
\end{equation*}
and the following inequality holds true for all $t \geq 1$: 
\begin{equation*}
\EE[\|\x_t - \x^*\|^2] \leq \|\x_0 - \x^*\|^2 + 2\sigma^2\sum_{j=1}^\infty \eta_j^2 \ < \ +\infty. 
\end{equation*}
Since $\|\x_t - \x^*\|^2 \geq 0$, we have $\|\x_t - \x^*\|^2 < +\infty$ almost surely. Therefore, $\EE[\|\psi_{t+1}\|^2 \mid \FCal_t] \leq \sigma^2\|\x_t - \x^*\|^2 < +\infty$. Then the law of large numbers for martingale differences yields that $\theta_{T+1}^{-1}(\sum_{t=1}^{T+1} \eta_t \psi_t) \rightarrow 0$ almost surely~\citep[Theorem~2.18]{Hall-2014-Martingale}. Furthermore, let $R_t = \sum_{j=1}^t \eta_j^2\|\xi_j\|^2$, then $R_t$ is a submartingale and 
\begin{equation*}
\EE[R_t] \ \leq \ \sigma^2\sum_{j=1}^t \eta_j^2 < \sigma^2\sum_{j=1}^\infty \eta_j^2 \ < \ +\infty. 
\end{equation*}
From Doob's martingale convergence theorem, $R_t$ converges to some random, finite value almost surely~\citep[Theorem~2.5]{Hall-2014-Martingale}. Putting these pieces together with Eq.~\eqref{inequality-noisy-OGD-absolute-last-main} yields that $\|\x_t - \x^*\|^2 \sim -\lambda c\tau_t \rightarrow -\infty$ almost surely, a contradiction. Therefore, we conclude that every neighborhood of $\XCal^*$ is recurrent.

\subsection{Proof of Theorem~\ref{Thm:noisy-OGD-absolute-average-rate}}
Since $\eta_t = c/\sqrt{t}$ for all $t \geq 1$, we have $\eta_t \rightarrow 0$ and $\eta_t \leq c$ for all $t \geq 1$. This implies that 
\begin{equation}\label{inequality-noisy-OGD-absolute-fourth}
\lambda\eta_{t+1} - \eta_{t+1}^2 \ \geq \ (\lambda - c)\eta_{t+1}. 
\end{equation}
Plugging Eq.~\eqref{inequality-noisy-OGD-absolute-fourth} into Eq.~\eqref{inequality-noisy-OGD-absolute-key} (cf. Lemma~\ref{Lemma:key-inequality-noisy-OGD}) yields that
\begin{equation*}
\|\x_{t+1} - \x^*\|^2 \ \leq \ \|\x_t - \x^*\|^2 - 2(\lambda - c)\eta_{t+1}\|\sv(\x_t)\|^2 + 2\eta_{t+1}(\x_t - \x^*)^\top\xi_{t+1} + 2\eta_{t+1}^2\|\xi_{t+1}\|^2. 
\end{equation*}
Using the same argument as in Theorem~\ref{Thm:noisy-OGD-absolute-last}, we have
\begin{equation}\label{inequality-noisy-OGD-absolute-fifth}
\EE[\|\x_{t+1} - \x^*\|^2 \mid \FCal_t] \ \leq \ \|\x_t - \x^*\|^2 - 2(\lambda - c)\eta_{t+1}\|\sv(\x_t)\|^2 + 2\eta_{t+1}^2\sigma^2.
\end{equation}
Taking the expectation of both sides of Eq.~\eqref{inequality-noisy-OGD-absolute-fifth} and rearranging yields that 
\begin{equation*}
\EE[\epsilon(\x_t)] \ \leq \ \frac{1}{2(\lambda - c)\eta_{t+1}}\left(\EE[\|\x_t - \x^*\|^2] - \EE[\|\x_{t+1} - \x^*\|^2]\right) + \frac{\eta_{t+1}\sigma^2}{\lambda - c}. 
\end{equation*}
Summing up the above inequality over $t=0, 1, \ldots, T$ and using $\eta_t = c/\sqrt{T+1}$ yields that 
\begin{equation*}
\EE\left[\sum_{t=0}^T\epsilon(\x_t)\right] \ \leq \ \frac{\|\x_0 - \x^*\|^2}{2(\lambda - c)\eta_1} + \sum_{t=1}^T \left(\frac{1}{\eta_{t+1}} - \frac{1}{\eta_t}\right)\frac{\EE[\|\x_t - \x^*\|^2]}{2(\lambda - c)} + \frac{\sigma^2}{\lambda-c}\left(\sum_{t=1}^{T+1} \eta_t\right). 
\end{equation*}
On the other hand, we have
\begin{equation*}
\EE[\|\x_{t+1} - \x^*\|^2] \ \leq \ \EE[\|\x_t - \x^*\|^2] + 2\eta_{t+1}^2\sigma^2. 
\end{equation*}
This implies that the following inequality holds for all $t \geq 1$: 
\begin{equation*}
\EE[\|\x_t - \x^*\|^2] \ \leq \ \|\x_0 - \x^*\|^2 + 2\sigma^2\left(\sum_{j=1}^t \eta_j^2\right) \ \leq \ \|\x_0 - \x^*\|^2 + 2\sigma^2 c^2\log(t+1). 
\end{equation*}
Therefore, we conclude that 
\begin{eqnarray*}
\EE\left[\sum_{t=0}^T\epsilon(\x_t)\right] & \leq & \frac{\|\x_0 - \x^*\|^2 + 2\sigma^2 c^2\log(T+1)}{2(\lambda-c)}\sum_{t=1}^T \left(\frac{1}{\eta_{t+1}} - \frac{1}{\eta_t}\right) + \frac{\|\x_0 - \x^*\|^2}{2(\lambda-c)\eta_1} + \frac{\sigma^2}{\lambda-c}\left(\sum_{t=1}^{T+1} \eta_t\right) \\
& \leq & \frac{\sqrt{T+1}(\|\x_0 - \x^*\|^2 + 2\sigma^2 c^2\log(T+1))}{2c(\lambda-c)} + \frac{\|\x_0 - \x^*\|^2}{2c(\lambda-c)} + \frac{\sigma^2 c\sqrt{T+1}}{\lambda-c}  \\
& = & O\left(\sqrt{T+1}\log(T+1)\right). 
\end{eqnarray*}
This completes the proof. 

\subsection{Proof of Theorem~\ref{Thm:noisy-OGD-absolute-last-rate}}
Using Lemma~\ref{Lemma:OGD-noisy-absolute-iterative}, we have
\begin{equation*}
\EE[\epsilon(\x_T)] \ \leq \ \EE[\epsilon(\x_t)] + \frac{1}{\lambda}\left(\sum_{j=t}^{T-1}\frac{\sigma_j^2}{\eta_{j+1}}\right) \ \leq \ \EE[\epsilon(\x_t)] + \frac{1}{\lambda\underline{\eta}}\left(\sum_{j=t}^{T-1} \sigma_j^2\right). 
\end{equation*}
Summing up the above inequality over $t = 0, 1, \ldots, T$ yields that 
\begin{equation*}
(T+1)\EE[\epsilon(\x_T)] \ \leq \ \sum_{t=0}^T \EE[\epsilon(\x_t)] + \frac{1}{\lambda}\left(\sum_{t=0}^{T-1} \sum_{j=t}^{T-1}\frac{\sigma_j^2}{\eta_{j+1}}\right) \ \leq \ \sum_{t=0}^T \EE[\epsilon(\x_t)] + \frac{1}{\lambda\underline{\eta}}\left(\sum_{t=0}^{T-1} (t+1)\sigma_t^2\right). 
\end{equation*}
On the other hand, the derivation in Theorem~\ref{Thm:noisy-OGD-absolute-average-rate} implies that 
\begin{eqnarray*}
\EE\left[\sum_{t=0}^T\epsilon(\x_t)\right] & \leq & \frac{\|\x_0 - \x^*\|^2}{2(\lambda-\overline{\eta})\eta_1} + \sum_{t=1}^T \left(\frac{1}{\eta_{t+1}} - \frac{1}{\eta_t}\right)\frac{\EE[\|\x_t - \x^*\|^2]}{2(\lambda-\overline{\eta})} + \frac{1}{\lambda-\overline{\eta}}\left(\sum_{t=1}^{T+1} \eta_t \sigma_t^2 \right) \\
& \leq & \frac{\|\x_0 - \x^*\|^2}{2(\lambda-\overline{\eta})\eta_1} + \sum_{t=1}^T \left(\frac{1}{\eta_{t+1}} - \frac{1}{\eta_t}\right)\frac{\EE[\|\x_t - \x^*\|^2]}{2(\lambda-\overline{\eta})} + \frac{\overline{\eta}}{\lambda-\overline{\eta}}\left(\sum_{t=1}^{T+1} \sigma_t^2 \right). 
\end{eqnarray*}
On the other hand, we have
\begin{equation*}
\EE[\|\x_{t+1} - \x^*\|^2] \ \leq \ \EE[\|\x_t - \x^*\|^2] + 2\eta_{t+1}^2\sigma_t^2. 
\end{equation*}
This implies that the following inequality holds for all $t \geq 1$: 
\begin{equation*}
\EE[\|\x_t - \x^*\|^2] \ \leq \ \|\x_0 - \x^*\|^2 + 2\overline{\eta}^2\left(\sum_{j=1}^t \sigma_j^2\right). 
\end{equation*}
Therefore, we conclude that 
\begin{eqnarray*}
\EE\left[\sum_{t=0}^T\epsilon(\x_t)\right] & \leq & \frac{\|\x_0 - \x^*\|^2 + 2\overline{\eta}^2(\sum_{t=1}^T \sigma_t^2)}{2(\lambda-\overline{\eta})}\sum_{t=1}^T \left(\frac{1}{\eta_{t+1}} - \frac{1}{\eta_t}\right) + \frac{\|\x_0 - \x^*\|^2}{2(\lambda-\overline{\eta})\eta_1} + \frac{\overline{\eta}}{\lambda-\overline{\eta}}\left(\sum_{t=1}^{T+1} \sigma_t^2\right) \\
& \leq & \frac{\|\x_0 - \x^*\|^2 + 2\overline{\eta}^2(\sum_{t=1}^T \sigma_t^2)}{2(\lambda-\overline{\eta})\underline{\eta}} + \frac{\|\x_0 - \x^*\|^2}{2(\lambda-\overline{\eta})\underline{\eta}} + \frac{\overline{\eta}}{\lambda-\overline{\eta}}\left(\sum_{t=1}^{T+1} \sigma_t^2\right) \\
& \leq & \frac{\|\x_0 - \x^*\|^2}{(\lambda-\overline{\eta})\underline{\eta}} + \left(1 + \frac{\overline{\eta}}{\underline{\eta}}\right)\frac{\overline{\eta}}{\lambda-\overline{\eta}}\left(\sum_{t=1}^{T+1} \sigma_t^2\right) \\
& \leq & \frac{\|\x_0 - \x^*\|^2}{(\lambda-\overline{\eta})\underline{\eta}} + \left(1 + \frac{\overline{\eta}}{\underline{\eta}}\right)\frac{\overline{\eta}}{\lambda-\overline{\eta}}\left(\sum_{t=1}^{T+1} (t+1)\sigma_t^2\right). 
\end{eqnarray*}
Putting these pieces together yields that 
\begin{eqnarray*}
\EE[\epsilon(\x_T)] & \leq & \frac{1}{T+1}\left[\sum_{t=0}^T \EE[\epsilon(\x_t)] + \frac{1}{\lambda\underline{\eta}}\left(\sum_{t=0}^{T-1} (t+1)\sigma_t^2\right)\right] \\ 
& \leq & \frac{1}{T+1}\left(\frac{\|\x_0 - \x^*\|^2}{(\lambda-\overline{\eta})\underline{\eta}} + \left(\frac{1}{\lambda\underline{\eta}} + \left(1 + \frac{\overline{\eta}}{\underline{\eta}}\right)\frac{\overline{\eta}}{\lambda-\overline{\eta}}\right)\left(\sum_{t=1}^{T+1} (t+1)\sigma_t^2\right) \right) \\
& \overset{\textnormal{Eq.~\eqref{condition:noisy-OGD-absolute}}}{=} & O(a(T)).  
\end{eqnarray*}
This completes the proof. 

\end{document}